\pdfoutput=1

\documentclass[11pt,reqno]{amsart}
\usepackage[T1]{fontenc}
\usepackage{microtype} %Add in final version (slow)

\usepackage{pkgfile}

\usepackage[margin=1in]{geometry}
\usepackage{tikz}

\usepackage[nocompress,noadjust]{cite}

\theoremstyle{plain} % just in case the style had changed
\newcommand{\thistheoremname}{}  
\newtheorem*{genericthm*}{\thistheoremname}
\newenvironment{namedthm}[1]
  {\renewcommand{\thistheoremname}{#1}%
   \begin{genericthm*}}
  {\end{genericthm*}}
\title{Upper tails and independence polynomials in random graphs}

\author[Bhattacharya]{Bhaswar B. Bhattacharya}
\address{B.\ B.\ Bhattacharya\hfill\break
	Department of Statistics\\ University of Pennsylvania\\ Philadelphia, PA 19104, USA.}
\email{bhaswar@wharton.upenn.edu}

\author[Ganguly]{Shirshendu Ganguly}
\address{S. Ganguly \hfill\break
	Department of Statistics\\ UC Berkeley \\ 
	Berkeley, California, CA 94720, USA.}
\email{sganguly@berkeley.edu}

\author[Lubetzky]{Eyal Lubetzky}
\address{E.\ Lubetzky\hfill\break
	Courant Institute\\
	New York University\\
	251 Mercer Street\\ New York, NY 10012, USA.}
\email{eyal@courant.nyu.edu}

\author[Zhao]{Yufei Zhao}
\address{Y.\ Zhao\hfill\break
	Department of Mathematics\\ 
	MIT\\ Cambridge, MA 02139, USA.}
\email{yufeiz@mit.edu}

\begin{document}

\begin{abstract}
	The upper tail problem in the Erd\H{o}s--R\'enyi random graph $G\sim\mathcal{G}_{n,p}$ 	
	asks to estimate the probability that the number of copies of a graph $H$ in $G$ exceeds its expectation by a  factor $1+\delta$.
	Chatterjee and Dembo showed that in the sparse regime of $p\to 0$ as $n\to\infty$ with $p \geq n^{-\alpha}$ for an explicit $\alpha=\alpha_H>0$, this problem reduces to a natural variational problem on weighted graphs, which was thereafter  asymptotically solved by two of the authors in the case where $H$ is a clique.
	
	Here we extend the latter work to any fixed graph $H$ and determine a function $c_H(\delta)$ such that, for $p$ as above and any fixed $\delta>0$, the upper tail probability is $\exp[-(c_H(\delta)+o(1))n^2 p^\Delta \log(1/p)]$, where $\Delta$  is the maximum degree of $H$.
	As it turns out, the leading order constant in the large deviation rate function, $c_H(\delta)$, is governed by the independence polynomial of $H$, defined as $P_H(x)=\sum i_H(k) x^k $ where $i_H(k)$ is the number of independent sets of size $k$ in $H$. For instance, if $H$ is a regular graph on $m$ vertices, then $c_H(\delta)$ is the minimum between  $\frac12 \delta^{2/m}$ and the unique positive solution of $P_H(x) = 1+\delta$.
\end{abstract}

\maketitle

\section{Introduction}

\subsection{The upper tail problem in the random graph}
Let $\cG_{n, p}$ be the Erd\H os--R\'enyi random graph on $n$ vertices with edge probability $p$, and let $X_H$ be the number of copies of a fixed graph $H$ in it.
The upper tail problem for $X_H$ asks to estimate the large deviation rate function given by
\[ R_H(n,p,\delta) := -\log \P\left(X_H \geq (1+\delta)\E[X_H]\right)\quad\mbox{ for fixed $\delta > 0$}\,,\]
a classical and extensively studied problem (cf.~\cite{JR02,Vu01,KV04,JOR04,JR04,Cha12,DK12b,DK12} and~\cite{Bol,JLR} and the references therein) which already for the seemingly basic case of triangles ($H=K_3$) is highly nontrivial and still not fully understood. It followed from works of
Vu~\cite{Vu01} and Kim and Vu~\cite{KV04} that\footnote{We write $f \lesssim g$ to denote $f = O(g)$; $f \asymp g$ means $f = \Theta(g)$;  $f \sim g$ means $f = (1+o(1))g$ and $f \ll g$ means $f = o(g)$.}
\[   n^2 p^2 \lesssim R_{K_3}(n,p,\delta)  \lesssim n^2 p^2\log(1/p)\]
(the lower bound used the so-called ``polynomial concentration'' machinery, whereas the upper bound follows, e.g., from the fact that an arbitrary set of $s \sim  \delta^{1/3} np$ vertices can form a clique in $\cG_{n,p}$ with probability $p^{\binom{s}2}=p^{O(n^2 p^2)}$, thus contributing $\binom{s}3 \sim \delta \binom{n}{3} p^3 = \delta \E[X_{K_3}]$ extra triangles). The correct order of the rate function was settled fairly recently by Chatterjee~\cite{Cha12}, and independently  by DeMarco and Kahn~\cite{DK12}, proving that $R_{K_3}(n,p,\delta) \asymp n^2 p^2 \log(1/p)$ for  $p \ge \frac{\log n} n$. This was later extended in~\cite{DK12b}  to cliques ($H=K_k$  for $k\geq 3$), establishing that
 for $p \ge n^{-2/(k-1)+\varepsilon}$ with  $\varepsilon>0$ fixed\footnote{More precisely,
DeMarco and Kahn~\cite{DK12b} showed that
$R_{K_k}\asymp \min\{n^2 p^{k-1} \log (1/p), n^kp^{\binom{k}{2}}\}$ for $p \ge n^{-2/(k-1)}$.},
\[
R_{K_k}(n,p,\delta) \asymp n^2 p^{k-1} \log (1/p)\,.
\]
The methods of~\cite{Cha12,DK12,DK12b} did not allow recovering the exact asymptotics of this rate function, and
in particular, one could ask, e.g., whether $R_{K_3}(n,p,\delta) \sim c(\delta) n^2 p^2 \log(1/p)$ with $c(\delta) = \frac12\delta^{2/3}$, as the aforementioned clique upper bound for it may suggest (recall its probability is $p^{\binom{s}2}$ for $s\sim \delta^{1/3} n p$).

Much progress has since been made in that front, propelled by the seminal work of Chatterjee and Varadhan~\cite{CV11} that introduced a large deviation framework for $\cG_{n,p}$ in the \emph{dense regime} ($0<p<1$ fixed) via the theory of graph limits  (cf.~\cite{CD10,CV11,LZ-dense} for more on the many questions still open in that regime). See the survey by Chatterjee~\cite{Cha16} on recent developments on this topic. In the sparse regime ($p\to 0$), in the absence of graph limit tools,  the understanding of large deviations for a fixed graph $H$ (be it even a triangle) remained very limited until a recent breakthrough paper of Chatterjee and Dembo~\cite{CD16} that reduced it to a natural variational problem in a certain range of $p$ (see Definition~\ref{def-var-prob} and Theorem~\ref{thm:CD} below). The third and fourth authors solved this variational problem asymptotically for triangles~\cite{LZ-sparse}, thereby yielding the following conclusion: for fixed $\delta > 0$, if $n^{-1/42} \log n \leq p = o(1)$, then
\begin{equation} \label{eq:tail-prob-K3}
R_{K_3}(n,p,\delta) \sim c(\delta) n^2p^2 \log(1/p)\quad\mbox{ where }\quad c(\delta)=\min \bigl\{ \tfrac12 \delta^{2/3}, \tfrac13 \delta\bigr\}\,,
\end{equation}
and we see that the clique construction from above gives the correct leading order constant if $\delta \ge 27/8$.
More generally, for every $k\ge 3$ there is an explicit $\alpha_k > 0$ so that, for fixed $\delta > 0$, if $n^{-\alpha_k} \le p = o(1)$,
\begin{equation} \label{eq:tail-prob-clique}
R_{K_k}(n,p,\delta) \sim c_k(\delta)n^2p^{k-1} \log(1/p) \quad\mbox{ where }\quad c_k(\delta)= \min \bigl\{ \tfrac12 \delta^{2/k}, \tfrac1k \delta\bigr\} \,.
\end{equation}

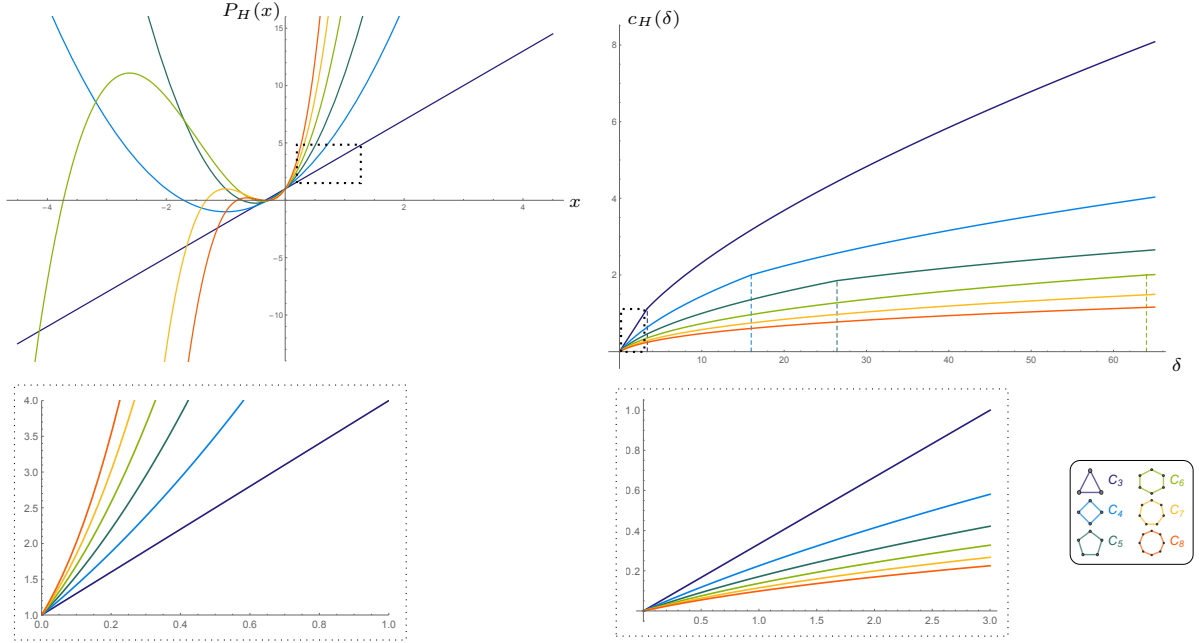
\begin{figure}\begin{center}
\vspace{-0.4cm}
\begin{tikzpicture}[font=\tiny,plotd/.style={draw=black,dotted}]
    \node (plot1) at (0,0) {
      \includegraphics[width=.45\textwidth]{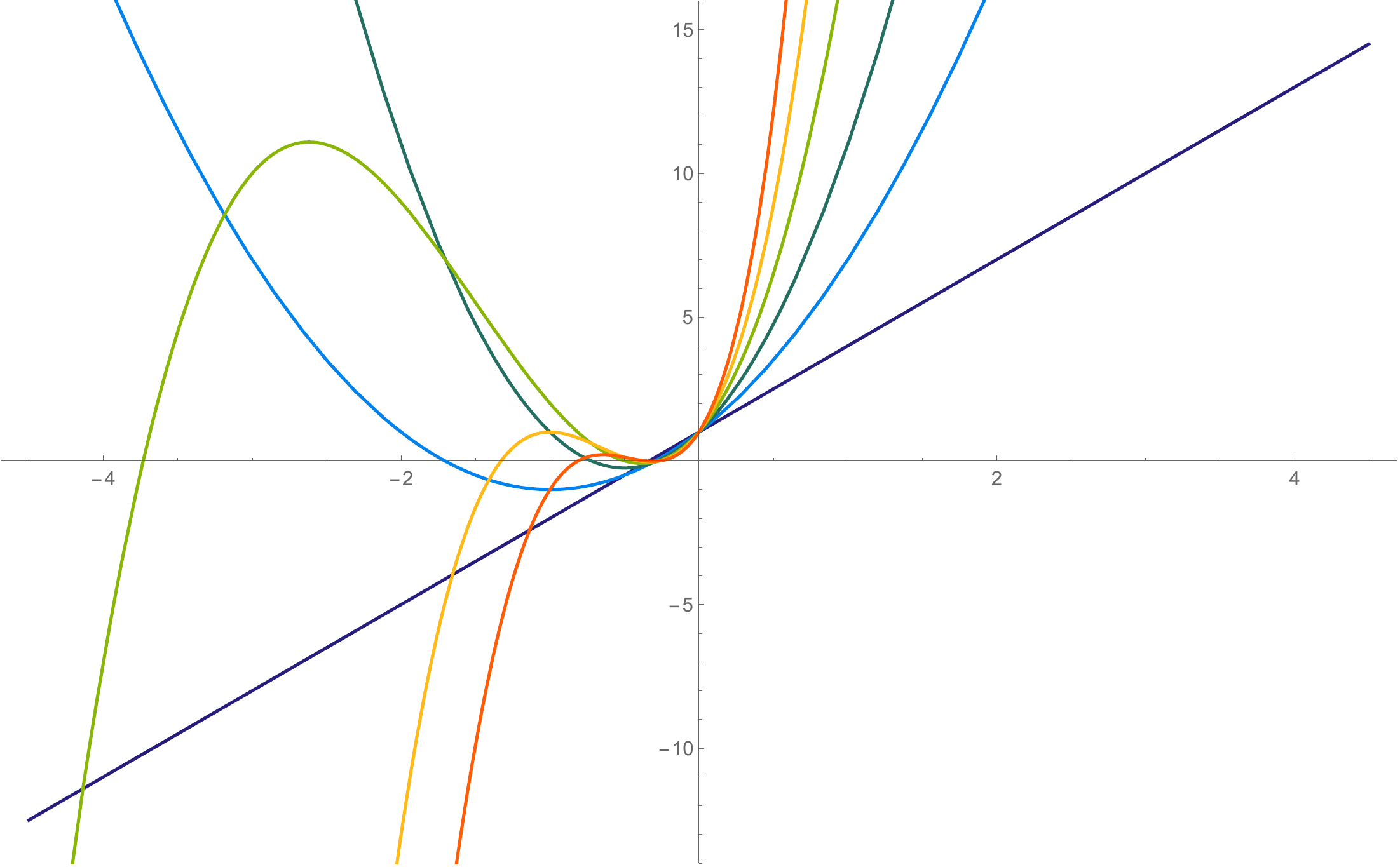}};
    \node (plot2) at (8cm,-0.1cm) {
      \includegraphics[width=.45\textwidth]{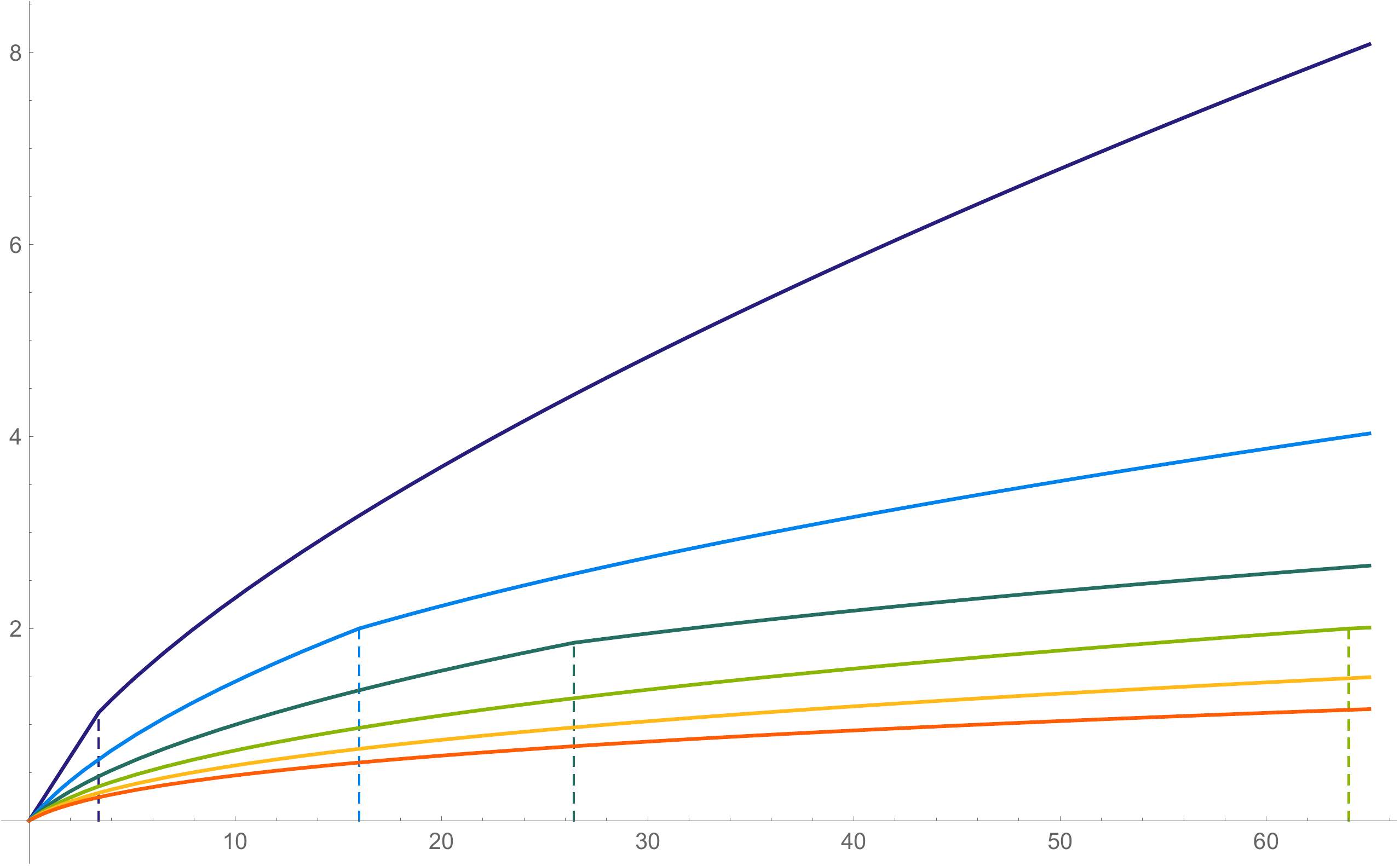}};
     \node (plot3) at (11.25,-4.3) {
        \includegraphics[width=0.1\textwidth]{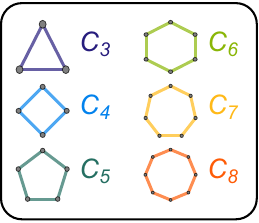}};
    \node[plotd] (plot4) at (7cm,-4.3cm) {
      \includegraphics[width=.3\textwidth]{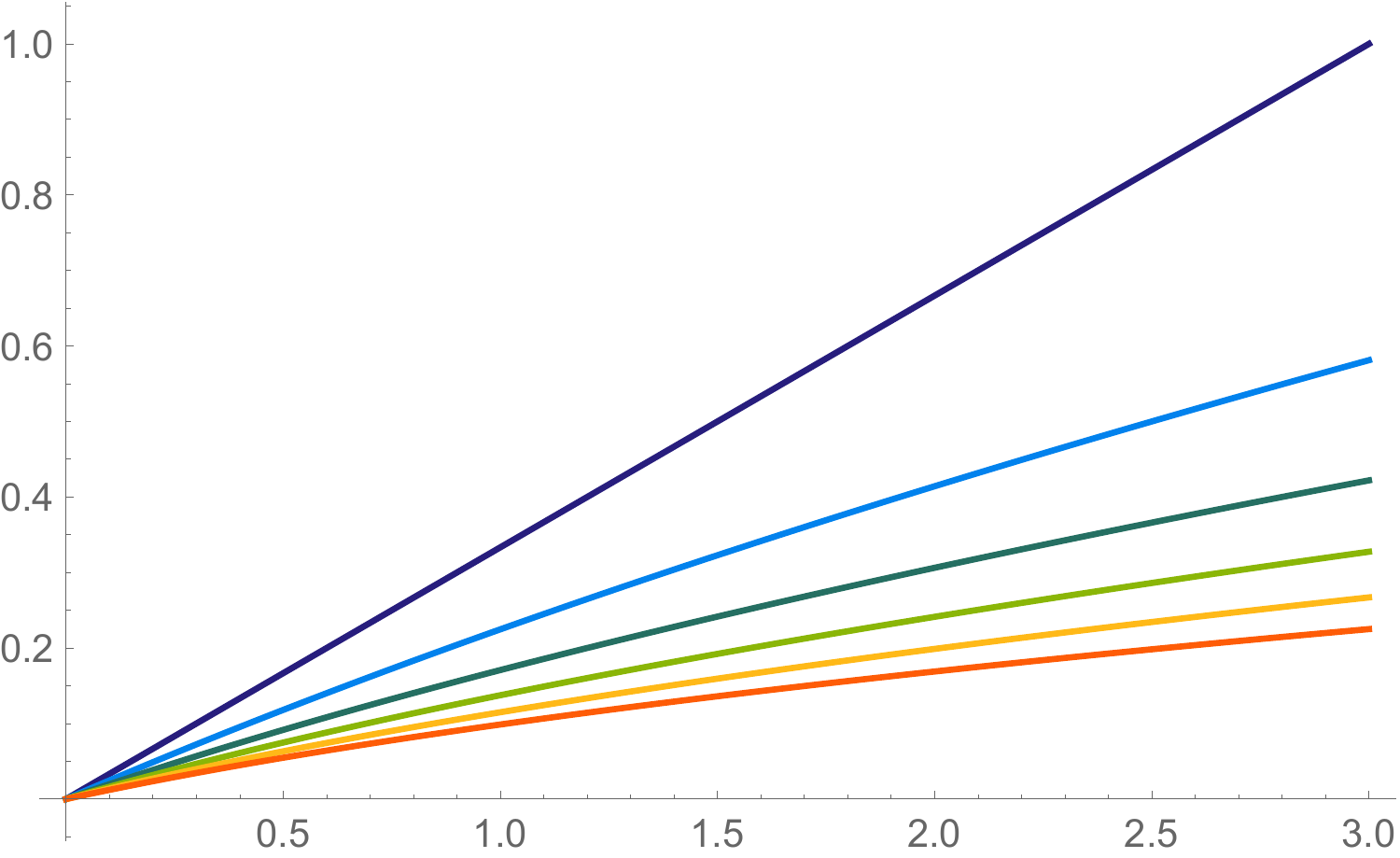}};
    \node[plotd] (plot5) at (-1cm,-4.3cm) {
      \includegraphics[width=.3\textwidth]{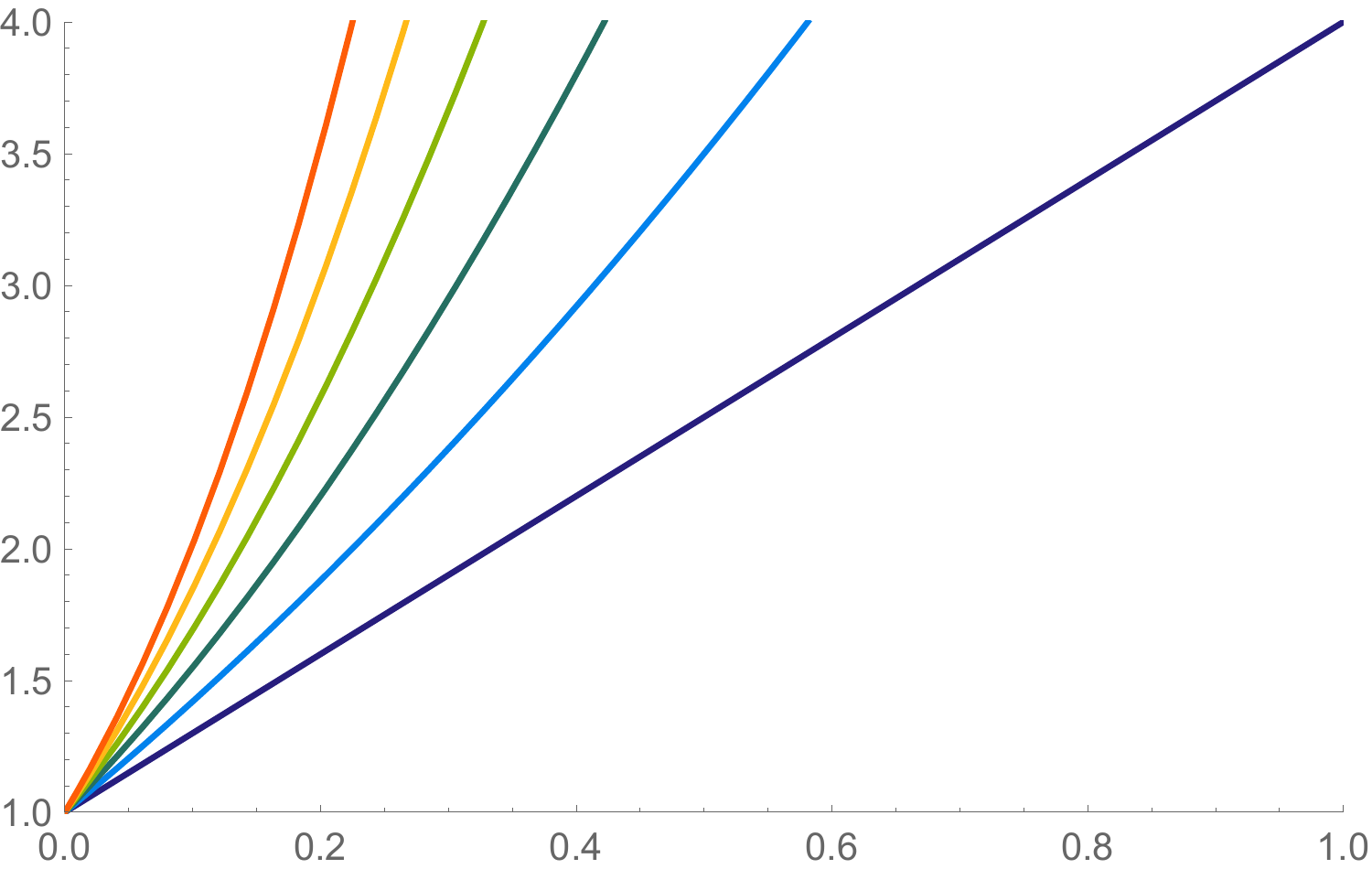}};
    \begin{scope}[shift={(plot1.south west)}]
    \draw[black,thick,dotted] (4.,2.52) -- (4.,3.03) -- (4.85,3.03) -- (4.85,2.52) -- cycle;
    \node at (7.7,2.25) {$x$};
    \node at (3.4,4.8) {$P_{H}(x)$};
    \end{scope}
    \begin{scope}[shift={(plot2.south west)}]
    \draw[black,thick,dotted] (.31,.37) -- (.31,.94) -- (.62,.94) -- (.62,.37) -- cycle;
    \node at (7.7,.2) {$\delta$};
    \node at (.75,4.8) {$c_{H}(\delta)$};
    \end{scope}
\end{tikzpicture}
\end{center}
\vspace{-0.35cm}
\caption{The leading order constant $c_H(\delta)$ for the upper tail rate function for $k$-cycles vs.\ their independence polynomials $P_H(x)$. Zoomed-in regions show $P_H(c_H(\delta))=1+\delta$.}
\vspace{-0.3cm}
\label{fig:ldp_cycles}
\end{figure}

For a \emph{general} fixed graph $H$ with maximum degree $\Delta \geq 2$ (when $\Delta=1$ the problem is nothing but the large deviation in the binomial variable corresponding to the edge-count) the order of the rate function was established up to a multiplicative $\log(1/p)$ factor by Janson, Oleszkiewicz, and Ruci\'nski~\cite{JOR04}. In the range $p \geq n^{-1/\Delta}$, their estimate (which involves a complicated quantity $M_H^*(n,p)$) simplifies into
\[
n^2 p^{\Delta} \lesssim R_H(n,p,\delta) \lesssim n^2 p^\Delta \log(1/p)
\]
(with constants depending on $H$ and on $\delta$). As a byproduct of the analysis of cliques in~\cite{LZ-sparse}, it was shown  \cite[Corollary~4.5]{LZ-sparse} that there is some explicit $\alpha_H > 0$ so that, for fixed $\delta > 0$, if $n^{-\alpha_H} \le p = o(1)$,
\[
R_H(n,p,\delta) \asymp n^2p^{\Delta} \log(1/p)\,,
\]
yet those bounds were not sharp already for the 4-cycle $C_4$.
Here we extend that work and determine the precise asymptotics of $R_H(n,p,\delta)$ for any fixed graph $H$ in the above mentioned range $n^{-\alpha_H} \le p = o(1)$ (as currently needed in the framework of~\cite{CD16}).  Solving the variational problem for a general $H$  requires significant new ideas  atop~\cite{LZ-sparse}, and turns out to involve the \emph{independence polynomial} $P_H(x)$ (see Fig.~\ref{fig:ldp_cycles}).

\begin{defn}[Independence polynomial] \label{def-ind-poly}
The \emph{independence polynomial} of $H$ is defined to be
\[
  P_H(x) := \sum_k i_H(k) x^k\,,
\]
where $i_H(k)$ is the number of $k$-element independent sets in $H$.
\end{defn}
\begin{defn}[Inducing on maximum degrees]\label{def-H*} For a graph $H$ with maximum degree $\Delta$, let $H^*$ be the induced subgraph of $H$ on all vertices whose degree in $H$ is $\Delta$.  (Note that $H^* = H$ if $H$ is regular.)
\end{defn}
Roots of independence polynomials were studied in various contexts (cf.~\cite{BHN03,BN05,CS07} and their references); here, the unique positive $x$ such that $P_{H^*}(x) = 1+\delta$ will, perhaps surprisingly, give the leading order constant (possibly capped at some maximum value if $H$ happens to be regular) of $R_H(n,p,\delta)$.

\subsection{Variational problem}
For graphs $G$ and $H$, denote by $\hom(H,G)$ the number of homomorphisms from $H$ to $G$ (a graph homomorphism is a map $V(H) \to V(G)$ that carries every edge of $H$ to an edge of $G$). The homomorphism density of $H$ in $G$ is defined as
$
t(H, G) := \hom(H, G) \, |V(G)|^{-|V(H)|}
$,
that is, the probability that a uniformly random map $V(H) \to V(G)$ is a homomorphism from $H$ to $G$.

Henceforth, we will work with $t(H, G)$ for $G \sim \cG_{n,p}$ instead of $X_H$ for convenience (the two quantities are nearly proportional, as the only possible discrepancies---non-injective homomorphisms from $H$ to $G$---are a negligible fraction of all homomorphisms when $G$ is sufficiently large and not too sparse).

Chatterjee and Dembo \cite{CD16} proved a non-linear large deviation principle, and in particular derived the exact asymptotics of the rate function for a general graph $H$ in terms of a variational problem.

\begin{defn}[Discrete variational problem]\label{def-var-prob}
Let $\sG_n$ denote the set of weighted undirected graphs on $n$ vertices with edge weights in $[0, 1]$, that is, if $A(G)$ is the adjacency matrix of $G$ then
\[
\sG_n =\left\{G_n:  A(G_n)=(a_{ij})_{1\leq i, j \leq n},\, 0\leq a_{ij}\leq 1,\, a_{ij} = a_{ji},\, a_{ii} = 0 \text{ for all } i, j\right\}.
\]
Let $H$ be a fixed graph with maximum degree $\Delta$. The variational problem for $\delta > 0$ and $0 < p < 1$ is
\begin{equation}
\phi(H, n, p, \delta):= \inf\left\{I_p(G_n) : G_n\in \sG_n \text{ with } t(H, G_n) \geq (1 + \delta)p^{|E(H)|}\right\}
\label{eq:disvar}
\end{equation}
where
\[
t(H, G_n):= n^{-|V(H)|}\sum_{1\leq i_1,\cdots,i_k\leq n}\prod_{(x, y)\in E(H)}a_{i_x i_y}
\]
is the density of (labeled) copies of $H$ in $G_n$, and $I_p(G_n)$ is the entropy relative to $p$, that is,
\[I_p(G) :=\sum_{1\leq i < j\leq n}I_p(a_{ij}) \quad \text{and} \quad I_p(x):=x\log\frac{x}{p}+(1-x)\log \frac{1-x}{1-p}\,.\]
\end{defn}
\begin{thm}[Chatterjee and Dembo~\cite{CD16}] \label{thm:CD}
Let $H$ be a fixed graph. There is some explicit $\alpha_H > 0$ such that for $n^{-\alpha_H}\leq p < 1$ and any fixed $\delta>0$,
\[
 \P\left(t(H, \cG_{n, p})\geq (1+\delta)p^{|E(H)|}\right)=\exp\bigl(-(1+o(1))\phi(H, n, p, \delta)\bigr)\,,
\]
where $\phi(H, n, p, \delta)$ is as defined in~\eqref{eq:disvar} and the $o(1)$-term goes to zero as $n \to \infty$.
\end{thm}

\begin{rmk*}
Recently Eldan~\cite{Eldan} improved the range of validity of the above theorem to $p \ge n^{-1/(6|E(H)|)}\log n$. 
\end{rmk*}

Thanks to this theorem, solving the variational problem $\phi(H,n,p,\delta)$ asymptotically would give the asymptotic rate function for $H$ when $n^{-\alpha_H} \le p = o(1)$
(as was done for $H = K_k$ in \cite{LZ-sparse}, yielding~\eqref{eq:tail-prob-clique}).

\begin{figure}\begin{center}
\begin{tikzpicture}[font=\normalsize,plotd/.style={draw=black,dotted}]
    \node (plot1) at (0,0) {
      \includegraphics[width=.25\textwidth]{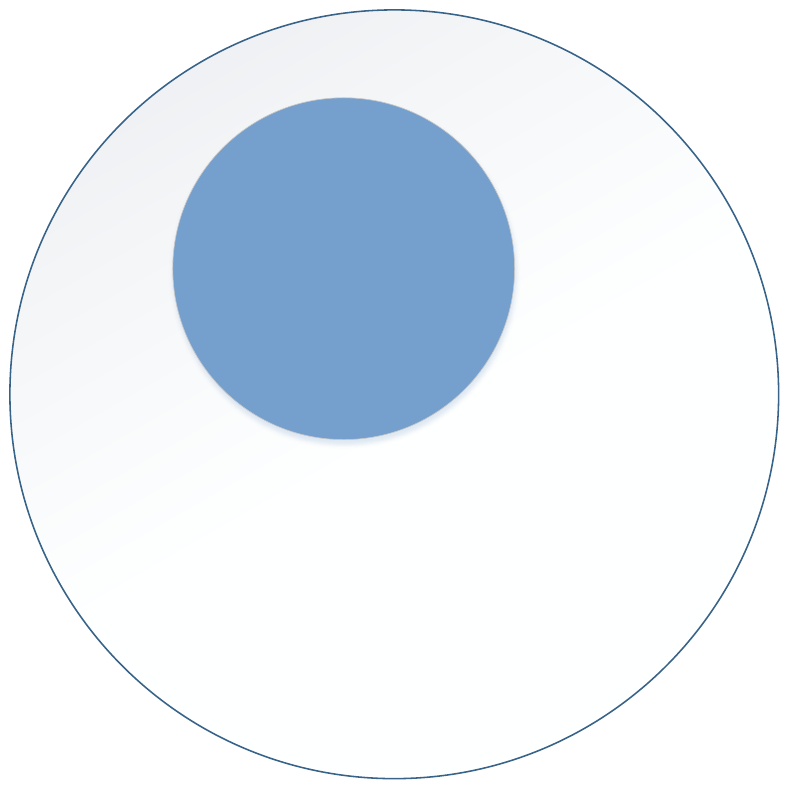}};
    \node (plot2) at (7cm,0cm) {
      \includegraphics[width=.26\textwidth]{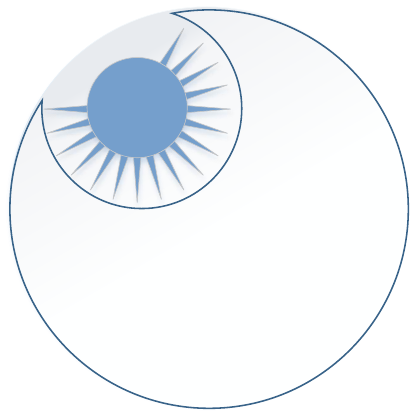}};
    \begin{scope}[shift={(plot1.south west)}]
    \node at (2,3) {\color{white}$1$};
    \node at (3,1.25) {$p$};
    \node at (2.4,3.9) {\tiny$s\sim\delta^{1/|V(H)|}p^{\Delta/2}n$};
    \end{scope}
    \begin{scope}[shift={(plot2.south west)}]
    \node at (3,1.5) {$p$};
    \node at (1.6,3.45) {\color{white}$1$};
    \node at (1.1,4.15) {\tiny$s\sim\theta p^\Delta n$};
    \end{scope}
\end{tikzpicture}
\end{center}
\vspace{-0.2cm}
\caption{Solution candidates for discrete variational problem (clique and anti-clique).}
\vspace{-0.25cm}
\label{fig:candidates}
\end{figure}
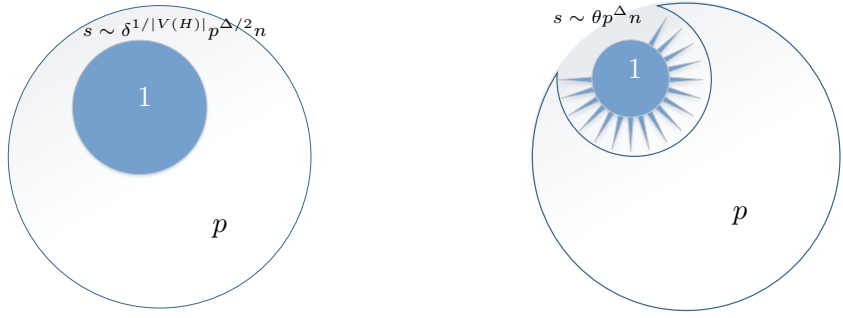

\subsection{Main Result} \label{sec:intro-main-result}

Let $H$ be a graph with maximum degree $\Delta = \Delta(H)$; recall that $H$ is \emph{regular} (or \emph{$\Delta$-regular}) if all its vertices have degree $\Delta$, and \emph{irregular} otherwise. Starting with a weighted graph $G_n$ with all edge-weights $a_{ij}$ equal to $p$, we consider the following two ways of modifying $G_n$ so it would satisfy the constraint  $t(H, G_n) \ge (1+\delta)p^{|E(H)|}$ of the variational problem \eqref{eq:disvar} (see Figure~\ref{fig:candidates}).
\begin{enumerate}[(a)]
	\item (Planting a clique) Set $a_{ij} = 1$ for all $1 \le i, j \le s$ for $s \sim  \delta^{1/|V(H)|} p^{\Delta/2} n$. This construction is effective only when $H$ is $\Delta$-regular, in which case it gives $t(H, G_n) \sim (1+\delta)p^{|E(H)|}$.
	\item (Planting an anti-clique) Set $a_{ij} = 1$ whenever $i \le s$ or $j \le s$ for $s \sim \theta p^\Delta n$ for $\theta = \theta(H,\delta)>0$ such that $P_{H^*}(\theta) = 1+\delta$, in which case $t(H, G_n) \sim (1+\delta)p^{|E(H)|}$.
\end{enumerate}
We postpone the short calculation that in each case $t(H, G_n) \sim (1+\delta)p^{|E(H)|}$ to \S\ref{sec:clique-anticlique}. Our main result (Theorem~\ref{thm:discrete-var} below) says that, for a connected graph $H$ and $n^{-1/\Delta}\ll p \ll 1$, one of these constructions has $I_p(G_n)$ that is within a $(1+o(1))$-factor of the optimum achieved by the variational problem \eqref{eq:disvar}.
For example, when $H = K_3$, the clique construction 
has $I_p(G_n) \sim \tfrac12 s^2 I_p(1) \sim \tfrac12 \delta^{2/3}n^2p^2 \log(1/p)$, while $P_{K_3}(x)=1+3x$ so $\theta = \delta/3$ and the anti-clique construction has $I_p(G_n) \sim s n I_p(1) \sim \tfrac13 \delta n^2 p^2 \log(1/p)$ (thus the clique wins if $\delta > 27/8$), exactly the bounds that were featured in~\eqref{eq:tail-prob-K3}. The following result extends~\cite[Theorems~1.1 and~4.1]{LZ-sparse} from cliques to the case of a general graph $H$. Recall $P_{H^*}(x)$ from Definitions~\ref{def-ind-poly} and~\ref{def-H*}.

\begin{thm}\label{thm:discrete-var}
Let $H$ be a fixed connected graph with maximum degree $\Delta \ge 2$. For any fixed $\delta>0$ and $n^{-1/\Delta} \ll p = o(1)$, the solution to the discrete variational problem~\eqref{eq:disvar} satisfies
\[\lim_{n \to \infty} \frac{\phi(H,n, p ,\delta)}{n^2p^\Delta\log(1/p)}=
\begin{cases}
\min\left\{\theta\,,\tfrac12 \delta^{2/|V(H)|}\right\}  &    \text{if $H$ is regular,} \\
\theta & \text{if $H$ is irregular,}
\end{cases}
\]
where $\theta=\theta(H,\delta)$ is the unique positive solution to $P_{H^*}(\theta)=1+\delta$.
\end{thm}

When combined with Theorem~\ref{thm:CD}, this yields the following conclusion for the upper tail problem.

\begin{cor} \label{cor:ldp}
Let $H$ be a fixed connected graph with maximum degree $\Delta \ge 2$.
There exists $\alpha_H > 0$ such that for $n^{-\alpha_H} \le p \ll 1$ the following holds. For any fixed $\delta>0$,
\begin{equation*}
\lim_{n \to \infty}\frac{-\log \P\left(t(H, \cG_{n, p})\geq (1+\delta)p^{|E(H)|}\right)}{n^2p^\Delta \log(1/p)}=
\begin{cases}
\min\left\{\theta\,,\tfrac12 \delta^{2/|V(H)|}\right\}  &    \text{if $H$ is regular,} \\
\theta & \text{if $H$ is irregular,}
\end{cases}
\end{equation*}
where $\theta=\theta(H,\delta)$ is the unique positive solution to $P_{H^*}(\theta)=1+\delta$.
\end{cor}

Observe that when $H$ is regular, there exists a unique
$\delta_0=\delta_0(H)>0$
such that\footnote{Indeed, $P_H(x)$ is increasing as it is a polynomial with nonnegative coefficients, so $\theta \leq \frac12 \delta^{2/|V(H)|}$ if and only if $1+\delta \le P_H(\frac12 \delta^{2/|V(H)|})$ (as $P_H(\theta) = 1 +\delta$). Since $P_H(x)$ is a polynomial of degree at most $|V(H)|/2$ (since $H$ is regular) and constant term 1, the function $f(\delta):= (P_H(\frac12\delta^{2/|V(H)|})-1)/\delta$ is decreasing for $\delta > 0$. We have $f(\delta) \to \infty$ as $\delta \to 0$ and $f(\delta) \le (2+o(1)) (\frac12 \delta^{2/|V(H)|})^{|V(H)|/2}/\delta \le 2^{1-|V(H)|/2} + o(1)$ as $\delta \to \infty$. So $f$ is decreasing and $f(\delta_0) = 1$ for some $\delta_0 > 0$, which proves the claim.}
\begin{equation}
   \label{eq-delta0-transition}
   \theta(H,\delta) \leq \tfrac12\delta^{2/|V(H)|} \quad\text{ if and only if } \quad \delta \leq \delta_0(H)\,.
 \end{equation}
That is, the leading order constant of $\phi(H,n,p,\delta)$ (giving the asymptotic upper tail) is governed by the anti-clique for $\delta \leq \delta_0$ and by the clique for $\delta \geq \delta_0$ (the above example of $H=K_3$ had $\delta_0 = 27/8$). 

We remark that our results extend (see Theorem~\ref{thm:disconnected}) to any \emph{disconnected} graph $H$. The interplay between different connected components can then cause the upper tail to be dominated not by an exclusive appearance of either the clique or the anti-clique constructions (as was the case for any \emph{connected} graph $H$, cf.~Theorem~\ref{thm:discrete-var}), but rather by an interpolation of these. See~\S\ref{sec:disconnected} for more details.

The assumption $p \gg n^{-1/\Delta}$ in Theorem~\ref{thm:discrete-var} is essentially tight in the sense that the upper tail rate function undergoes a phase transition at that location~\cite{JOR04}: it is of order $n^{2+o(1)}p^{\Delta}$ for $p \geq n^{-1/\Delta}$, and
below that threshold it becomes a function (denoted $M_H^*(n,p)$ in~\cite{JOR04}) depending on all subgraphs of $H$.
In terms of the discrete variational problem~\eqref{eq:disvar}, again this threshold marks a phase transition, as the anti-clique construction ceases to be viable for $p \ll n^{-1/\Delta}$ (recall that $s \sim \theta p^\Delta n$ in that construction). Still, as in~\cite[Theorems~1.1 and~4.1]{LZ-sparse}, our methods show that if $H$ is \emph{regular} and $n^{-2/\Delta} \ll p \ll n^{-1/\Delta}$, the solution to the variational problem is  $(1+o(1))\frac12 \delta^{2/|V(H)|}$ (i.e., governed by the clique construction).

Several of the tools that were developed here to overcome the obstacles in extending the analysis of~\cite{LZ-sparse} to general graphs (arising already for the 4-cycle) may be of independent interest and find other applications,  e.g., the crucial use of adaptively chosen degree-thresholds (see~\S\ref{sec:triangle-4cycle} for details).

One can ask to describe the random graph conditioned on having $H$-density at least $(1+\delta)p^{|E(H)|}$. Informally, our results suggest that the conditioned graph measure exhibits ``localization'', that is, the excess copies of $H$ are located in a microscopic part of the graph. In particular, we expect that it behaves like a typical  graph along with a randomly planted clique or a complete bi-partite graph (anti-clique) and the nature of the planted structure undergoes a phase transition in $\delta$, when $H$ is regular. However, unlike the dense setting~\cite{CV11,LZ-dense} where one can characterize the conditioned random graph with respect to the cut metric on graphs, we do not know of a good way to formalize the notion of being ``close to'' a planted clique or a planted anti-clique in the sparse setting.

\subsection{Examples} We now demonstrate the solution of the variational problem~\eqref{eq:disvar}, as provided by  Theorem~\ref{thm:discrete-var}, for various families of graphs (adding to the previously known~\cite{LZ-sparse} case of cliques, cf.~\eqref{eq:tail-prob-clique}).

\begin{example}[$k$-cycle: $H=C_k$]\label{ex-cycles}
It is easy to verify that $P_{C_k}(x)$ satisfies the recursion\footnote{By the definition of the independence polynomial, for any graph $H$ and vertex $v$ in it, $P_{H}(x)= P_{H_1}(x) + x P_{H_2}(x)$, where $H_1$ is obtained from $H$ by deleting $v$ and $H_2$ is obtained from $H$ by deleting $v$ and all its neighbors.}
\[ P_{C_k}(x) = P_{C_{k-1}}(x) + x P_{C_{k-2}}(x)\,,\quad P_{C_2}(x) = 2x+1 \,,\quad P_{C_3}(x) = 3x+1\,.\]
For instance, $P_{C_4}(x) = 2x^2+4x+1$ and $P_{C_5}(x)=5x^2+5x+1$; by Theorem~\ref{thm:discrete-var}, if $n^{-1/2}\ll p \ll 1$,
\begin{align*} \phi(C_4, n, p, \delta)
&\sim \min\left\{ \theta(C_4,\delta)\,,\tfrac12 \delta^{1/2}\right\} n^2p^2\log(1/p) & \mbox{ for }\quad&
\theta(C_4,\delta) = -1+\sqrt{1+\tfrac12 \delta}\,,\\
\phi(C_5, n, p, \delta)&\sim \min\left\{
\theta(C_5,\delta)\,,
\tfrac12 \delta^{2/5}  \right\}n^2p^2\log(1/p) & \mbox{ for }\quad&
\theta(C_5,\delta) = -\tfrac12+\tfrac12\sqrt{1+\tfrac45 \delta}\,.
\end{align*}
For general $k$, with square brackets denoting extraction of coefficients, $[x]P_{C_k}(x)=k$ (more generally,  $[x]P_H(x)=|V(H)|$ for any $H$), while the closely related recursion for Chebyshev's polynomials yields
\begin{equation}\label{eq-P_Ck} P_{C_k}(x)  =2^{1-k}\sum_{j=0}^{\lfloor k/2\rfloor} \binom{k}{2 j} (1+4x)^j\,,\end{equation}
and so $[x^2]P_{C_k}(x)=\frac12k(k-3)$; e.g., for any $k\geq 4$, the behavior of $\theta(C_k,\delta)$ for small $\delta$ (see Fig.~\ref{fig:ldp_cycles}) is
\[ \theta(C_k,\delta) = \tfrac{1}{k-3}\left(-1+\sqrt{1+2\delta(k-3)/k}\right) + O(\delta^3) =\tfrac1k \delta + \tfrac{3-k}{2k^2}\delta^2 + O(\delta^3)\,. \]
Finally, observe that for even $k$ we can write~\eqref{eq-P_Ck} as $ P_{C_k}(x)=\left[\frac12(\sqrt{1+4x}+1)\right]^k+\left[\frac12(\sqrt{1+4x}-1)\right]^k$,
and deduce that the value of $P_{C_k}(\frac12\delta^{2/k})$ for $\delta=2^k$ is simply $P_{C_k}(2)=2^k+1= 1+\delta$. Thus, by the remark following Corollary~\ref{cor:ldp}, the transition addressed in~\eqref{eq-delta0-transition} occurs at $\delta_0(C_k)=2^k$ for even $k$; e.g.,
\begin{align}
\lim_{n \to \infty}\frac{\phi(C_4, n, p, \delta)}{n^2p^2\log(1/p)}
&=
\begin{cases}
-1+\sqrt{1+\tfrac12 \delta}  &      \text{if } \delta < 16 \,,\\
\tfrac12 \sqrt{\delta}  &  \text{if } \delta \geq 16 \,.
\end{cases}
\label{eq:ex-C4}
\end{align}
As mentioned above, $H=C_4$ is the simplest graph for which the arguments in~\cite{LZ-sparse} did not give sharp bounds on $\phi(H,n,p,\delta)$, and its treatment is instrumental for the analysis of general graphs (see~\S\ref{sec:4cycle}).
\end{example}

\begin{example}
  [Binary tree] Letting $T_h$ denote the complete binary tree of height $h$ ($|V(T_h)|=2^h-1$), observe that, by counting independent sets excluding/including the root,  $P_{T_h}(x)$ satisfies the recursion
\[ P_{T_h}(x) = P_{T_{h-1}}(x)^2 + x P_{T_{h-2}}(x)^4\,,\quad P_{T_0}(x) = 1\,,\quad P_{T_1}(x) = x+1 \,.\]
The polynomial $P_{T_h^*}(x)$ restricts us to independent sets of $T_h$ where all degrees are $\Delta$, and therefore
\[ P_{T_h^*}(x) = P_{T_{h-2}}(x)^2\,,\]
as the restriction excludes precisely the root and leaves.
(More generally, for the $b$-ary tree ($b\geq 2$) one has $P_{T_h}(x)=P_{T_{h-1}}(x)^{b} + x P_{T_{h-2}}(x)^{b^2}$, and $P_{T_h^*}(x) = P_{T_{h-2}}(x)^{b}$.)

For instance, the binary tree on 15 vertices, $T_4$, has
\[ P_{T_4^*}(x) = P_{T_2}(x)^2 = x^4 + 6x^3 + 11x^2 + 6x + 1\,,\]
and solving $P_{T_4^*}(\theta) = 1+\delta$, we obtain, by Theorem~\ref{thm:discrete-var}, that for any $n^{-1/3}\ll p \ll 1$,
\begin{equation} \label{eq:ex-T4}
\lim_{n \to \infty}\frac{\phi(T_{4}, n, p, \delta)}{n^2 p^3 \log(1/p)}= -\tfrac32 + \tfrac12 \sqrt{5+ 4 \sqrt{1+\delta}}\,.
\end{equation}
For general $h$, we can for instance deduce from the recurrence above (and the facts $[x]P_H[x]=|V(H)|$
and $[x]P_{H^*}(x)=\#\{v: \deg(v)=\Delta\}$) that 
 $ [x]P_{T_h^*}(x)=2^{h-1}-2$ and $ [x^2]P_{T_h^*}(x) = 2^{2h-3}-7\cdot 2^{h-2}+7$ for any $h\geq 3$, using which it is easy to write $\theta(T_h,\delta)$ explicitly up to an additive $O(\delta^3)$-term.
\end{example}

\begin{example}[Complete bipartite: $H = K_{k, \ell}$ for $k \geq \ell$]
In case $k>\ell$ we have $P_{K_{k,\ell}^*}(x) = (1+x)^{\ell}$ as we only count independent sets in the $k$-regular side (of size $\ell$); thus, by Theorem~\ref{thm:discrete-var}, for $n^{-1/k}\ll p \ll 1$,
\begin{equation}\label{eq:ex-Kkl-equal}
\lim_{n \to \infty}\frac{\phi(K_{k,\ell}, n, p, \delta)}{n^2p^k\log(1/p)}= (1+\delta)^{1/\ell}-1\,.
\end{equation}
If $k=\ell$, the coefficients of $x^j$ ($j\geq 1$) are doubled, so $P_{K_{k,\ell}}(x)=2(1+x)^k - 1$ and for $n^{-1/k}\ll p \ll 1$,
\begin{equation}\label{eq:ex-Kkl-unequal}
\lim_{n \to \infty}\frac{\phi(K_{k,\ell}, n, p, \delta)}{n^2p^k\log(1/p)}= \min\left\{ \left(1+\tfrac12 \delta\right)^{1/k}-1\,,\tfrac12 \delta^{1/k}\right\}\,.
\end{equation}
\end{example}

\section{Clique and anti-clique constructions}
\label{sec:clique-anticlique}

We prove the claim at the beginning of \S\ref{sec:intro-main-result}, which gives an upper bound to the discrete variational problem $\phi(H, n, p, \delta)$. It is obtained by planting a clique of an anti-clique of appropriate size (see Figure~\ref{fig:candidates}).

\begin{ppn}\label{ppn:construction}
Let $H$ be a graph with maximum degree $\Delta$. Let $\delta > 0$ and $\theta=\theta(H,\delta)$ the unique positive solution to $P_{H^*}(\theta)=1+\delta$. 
\begin{enumerate}
\item[(a)] (Clique) If $H$ is connected and $\Delta$-regular and $n^{-2/\Delta} \ll p \ll 1$, then
\[
\phi(H,n,p,\delta) \le \bigl(\tfrac12 \delta^{2/|V(H)|} + o(1)\bigr) n^2 p^{\Delta} \log(1/p)\,.
\]
\item[(b)] (Anti-clique) For any graph $H$ with maximum degree $\Delta$ (not necessarily connected or regular), if $n^{-1/\Delta} \ll p \ll 1$, then
\[
	\phi(H,n,p,\delta) \le (\theta + o(1)) n^2 p^{\Delta} \log(1/p)\,.
\]
\end{enumerate}
\end{ppn}

\begin{proof}
(a) Let $G$ be a weighted graph on $n$ vertices with adjacency matrix $(a_{ij})_{1\le i,j \le n}$. Starting with all weights set to $p$, modify $G$ by setting $a_{ij} = 1$ whenever $i,j \le s$ for some integer $s \sim \delta^{1/|V(H)|} p^{\Delta/2} n$ to be decided. Then $I_p(G) \sim \tfrac12 s^2 I_p(1) \sim \frac12 \delta^{2/|V(H)|} p^\Delta \log(1/p)$. We will show that $s \sim \theta p^\Delta n$ implies that $t(H, G) \sim (1+\delta) p^{|E(H)|}$, so that an appropriately chosen $s \sim \delta^{1/|V(H)|} p^{\Delta/2} n$ would give $t(H, G) \ge (1+\delta) p^{|E(H)|}$, thereby showing the claimed upper bound on $\phi(H, n, p, \delta)$.

By summing over the subset of vertices of $H$ that get mapped to $\{1, \dots, s\}\subseteq V(G)$, we find
\begin{align*}
t(H, G) 
&\sim \sum_{S \subseteq V(H)} \Bigl(\frac s n\Bigr)^{|S|} \Bigl(1-\frac s n\Bigr)^{|V(H)| - |S|} p^{|E(H)|-|E(H[S])|}\,
\\
&\sim  \sum_{S \subseteq V(H)} \Bigl(\delta^{1/|V(H)|} p^{\Delta/2} \Bigr)^{|S|} p^{|E(H)|-|E(H[S])|}
\sim (1+\delta)p^{|E(H)|}.
\end{align*}
Here $H[S]$ denotes the subgraph of $H$ induced by $S$. 
The first estimate hides a $1+o(1)$ factor coming from the negligible fraction of maps $V(H) \to V(G)$ that send two adjacent vertices of $H$ to the same vertex in $G$.
For the final estimate, note that since $H$ is $\Delta$-regular and connected, we have $\Delta |S|/2 > |E(H[S])|$ for all $\emptyset \ne S \subsetneq V(H)$, and in such cases the corresponding term in the summation above is $o(p^{|E(H)|})$. The only non-negligible terms are $S = \emptyset$ and $S = V(H)$, which make up the final estimate $(1+\delta)p^{|E(H)|}$.

(b) Let $G$ be a weighted graph on $n$ vertices with adjacency matrix $(a_{ij})_{1\le i,j \le n}$. Starting with all weights set to $p$, modify $G$ by setting $a_{ij} = 1$ whenever $i \le s$ or $j \le s$ for some integer $s \sim \theta p^\Delta n$. Then $I_p(G) \sim s n I_p(1) \sim \theta n^2p^\Delta \log(1/p)$. As earlier, it remains to show  $t(H, G) \sim (1+\delta) p^{|E(H)|}$.

In computing $t(H, G)$, by summing over the subset of vertices of $H$ that get mapped to $\{1, \dots, s\} \subseteq V(G)$, we find
\begin{align*}
t(H, G) 
&\sim \sum_{S \subseteq V(H)} \left(\frac s n\right)^{|S|} \left(1-\frac s n\right)^{|V(H)| - |S|} p^{|E(H[V\setminus S])|}
\\
&\sim \sum_{S \subseteq V(H)} \theta^{|S|} p^{\Delta |S| + |E(H[V\setminus S])|}\,
\\
&\sim \sum_{\substack{S \text{ indep.\ set} \\ \text{of } H^*}} \theta^{|S|} p^{|E(H)|} = P_{H^*}( \theta ) p^{|E(H)|} = (1+\delta) p^{|E(H)|}\,,
\end{align*}
as any $S \subseteq V(H)$ that is not an independent set of $H^*$ satisfies $\Delta|S| + |E(H[V\setminus S])| > |E(H)|$ and hence contributes negligibly to the sum.
\end{proof}

\section{The graphon formulation of the variational problem}
\label{sec:graphon}

Following~\cite{LZ-sparse}, we will analyze a continuous version of the discrete variational problem~\eqref{eq:disvar}, which has the advantage of having no dependence on $n$. Recall that a \emph{graphon} is a  symmetric measurable function $W : [0, 1]^2\rightarrow [0, 1]$ (where symmetric means $W(x,y) = W(y,x)$). In the continuous version of~\eqref{eq:disvar}, $W$  replaces the edge-weighted graph $G_n$, as the latter can be viewed as a discrete approximation of a graphon (see, e.g.,~\cite{BCLSV08,BCLSV12,Lov12,LS06} for more on graph limits). We write
$
\E[f(W)] := \int_{[0,1]^2} f(W(x,y))\, \mathrm d x \mathrm d y
$.

\begin{defn}[Graphon variational problem]
For $\delta>0$ and $0<p<1$, let
\begin{equation}
\phi(H, p, \delta):=\inf\Bigl\{\tfrac12 \E[I_p(W)] : \text{graphon $W$ with } t(H, W)\geq (1+\delta)p^{|E(H)|}\Bigr\},\
\label{eq:continuous}
\end{equation}
where
\[
t(H, W):=\int_{[0, 1]^{|V(H)|}}\prod_{(i, j)\in E(H)}W(x_i, x_j) \, \mathrm d x_1\mathrm d x_2\cdots \mathrm d x_{|V(H)|}\,.
\]
\end{defn}

For example, for $H = K_3$ we wish to minimize  
\[
\E[I_p(W)]:=\int_{[0, 1]^2}I_p(W(x, y)) \, \mathrm d x\mathrm d y
\]
over all graphons $W$ whose triangle density 
\[
t(K_3, W) = \int_{[0,1]^3} W(x,y)W(x,z)W(y,z) \, \mathrm d x \mathrm d y \mathrm d z
\]
is at least $(1+\delta)p^3$.

The solution of the graphon variational problem is given by the following two theorems. Recall Definitions~\ref{def-ind-poly} and~\ref{def-H*} for the independence polynomial $P_H(x)$ and the subgraph $H^*$ of $H$ induced by its maximum degree vertices.

\begin{thm} \label{thm:main-graphon}
Let $H$ be a connected $\Delta$-regular graph. Fix $\delta>0$ and let $\theta=\theta(H,\delta)$ be the unique positive solution to $P_H(\theta)=1+\delta$. Then
\[
\lim_{p\rightarrow 0} \frac{\phi(H, p, \delta)}{p^\Delta\log(1/p)}
=
\begin{cases}
\min\left\{\theta\,,\tfrac12 \delta^{2/|V(H)|}\right\}  &    \text{if $H$ is regular,} \\
\theta & \text{if $H$ is irregular.}
\end{cases}
\]
\end{thm}

Let us deduce the continuous version, our main theorem Theorem~\ref{thm:discrete-var}, from the discrete analog.

\begin{lem} \label{lem:disccont} 
For any $H, p, n, \delta$, we have
$\phi( H, p, \delta)\le n^{-2}\phi(H, n, p, \delta)$.
\end{lem}

\begin{proof}
Given weighted graph $G_n \in \sG_n$ with adjacency matrix $(a_{ij})_{1\le i,j \le n}$, form a graphon $W^{G_n} $ as follows: divide $[0,1]$ into $n$ equal-length intervals $I_1,I_2, \ldots, I_n$ and set $W^{G_n}(x,y)=a_{ij}$ if $x\in I_i, y \in I_j$ and $i\neq j$, and $W^{G_n}(x, y)=p$ if $x,y \in I_i$ for some $i$. The lemma follows after noting that $t(H,G_n) \le t(H,W^{G_n})$ and
$I_{p}(W^{G_n})=n^{-2} I_p(G_n)$ (diagonal entries contribute $0$ to $I_{p}(W^{G_n})$).
\end{proof}

\begin{proof}[Proof of Theorem~\ref{thm:discrete-var} assuming Theorem~\ref{thm:main-graphon}] 
The upper bound to $\phi(H,n,p,\delta)$ is given by Proposition~\ref{ppn:construction}. The lower bound follows by Theorem~\ref{thm:main-graphon} and Lemma~\ref{lem:disccont}.
\end{proof}

It remains to prove Theorem~\ref{thm:main-graphon}, which is the goal for the rest of the paper. Note that the upper bound to $\phi(H,p,\delta)$ follows by Proposition~\ref{ppn:construction} and Lemma~\ref{lem:disccont}. Alternatively, consider the graphon analogs of the clique and anti-clique constructions (see Figure~\ref{fig:graphon-clique-anticlique}):
\begin{enumerate}
	\item[(a)] (Clique graphon) Modify the constant graphon $W \equiv p$ by setting $W(x,y) = 1$ whenever $x,y \in [0,a]$ for $a \sim \delta^{1/|V(H)|}$.
	\item[(a)] (Anticlique graphon) Modify the constant graphon $W \equiv p$ by setting $W(x,y) = 1$ whenever $\min\{x,y\} \in [0,b]$ for $b \sim \theta p^{\Delta}$.
\end{enumerate}
By essentially the same calculations as in \S\ref{sec:clique-anticlique}, we have $t(H, W) \sim (1+\delta)p^{|E(H)|}$ for both graphons above. Calculating their entropies yields the claimed upper bounds to $\phi(H, p, \delta)$.

\begin{figure}
\vspace{-0.1cm}
\includegraphics[scale=.65]{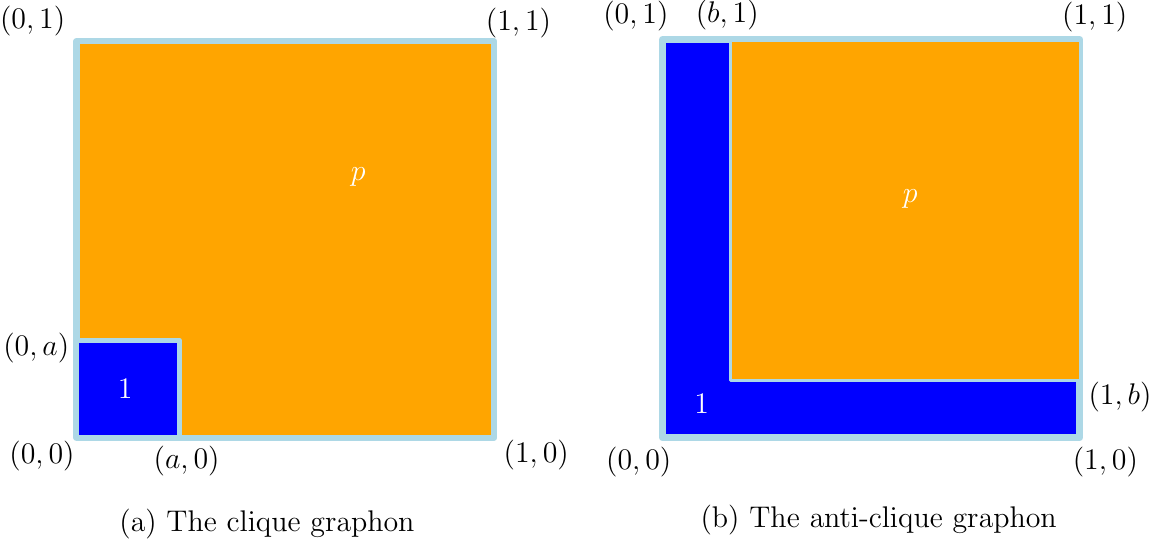}
\vspace{-0.25cm}
\caption{Solution candidates for the graphon variational problem ($a\asymp p^{\Delta/2}$ and $b \asymp p^\Delta$).}
\label{fig:graphon-clique-anticlique}
\vspace{-0.25cm}
\end{figure}

\section{Preliminaries}
\label{sec:prelim}
In this section, we recall various relevant estimates from~\cite{LZ-sparse}, used there to solve the continuous variational problem~\eqref{eq:continuous} for the case of cliques. A key inequality used both in~\cite{LZ-sparse} and in its prequel dealing with dense graphs~\cite{LZ-dense} is the following generalization of H\"older's inequality \cite[Theorem~2.1]{Fin92} (closely related to the Brascamp--Lieb inequalities~\cite{BL76}).
\begin{thm}[Generalized H\"older's inequality] 
	\label{thm:holder} 
	Let $\mu_1,\mu_2,\ldots \mu_n$ be probability measures on $\Omega_1,\ldots \Omega_n$ resp., and let $\mu = \prod_{i=1}^n\mu_i$.
	Let $A_1 \ldots A_m$ be non-empty subsets of $[n] = \{1,\ldots n\}$ and for $A\subseteq [n]$ put $\mu_A=\prod_{j\in A}\mu_j $ and $\Omega_A=\prod_{j\in A} \Omega_j$.
	Let $f_i\in L^{p_i}(\Omega_{A_i},\mu_{A_i})$ for each $i\in [m]$, and further suppose that $\sum_{i\colon A_i\ni j}(1/ p_i)\le 1$ for all $j\in [n]$.
	Then
	\[
		\int \prod_{i=1}^m \left|f_i \right| \mathrm d\mu \le \prod_{i=1}^m \left(\int \left|f_i\right|^{p_i}\mathrm d\mu_{A_i}\right)^{1/p_i}\,.
	\]
\end{thm}
Note that, in particular, if every element of $[n]$ is contained in at most $\Delta$ many sets $A_j$, then one can take $p_i = \Delta$ for all $i \in [m]$, giving the inequality $\int f_1\dots f_m \mathrm d \mu \le \prod_{i=1}^m \big(\int \left|f_i\right|^{\Delta}\mathrm d\mu_{A_i}\big)^{1/\Delta}$.

We will mostly be applying the generalized H\"older's inequality with each $A_i$ being a two-element set corresponding to an edge of a graph, and all $p_i$'s set to the maximum degree of the graph. Though there are a few tricky cases where it will be important to use non-uniform $p_i$'s.

Let $H$ be any graph with maximum degree $\Delta$, and let $W$ be a graphon with $t(H, W)\geq (1+\delta)p^{|E(H)|}$. Since $I_p$ is convex and decreasing from $0$ to $p$ and increasing from $p$ to $1$, we may assume $W\geq p$, i.e.,
\begin{equation}
U:= W -p
\quad\text{satisfies} \quad
0 \leq U \leq 1-p 
\quad\text{and}\quad
t(H,p+U) \geq (1+\delta)p^{|E(H)|}\,.
\label{eq:Wgeqp}
\end{equation}
For $b\in(0,1]$, define the set $B_b$ of points $x$ with high \emph{normalized degree} $d(x)$ in $U$ by
\begin{equation}
B_b = B_{b}(U):=\{x:d_U(x)\ge b\}\,,\quad\text{ where }\quad d(x) = d_U(x):=\int_0^1 U(x, y)\,\mathrm d y\,.
\label{eq:highdeg}
\end{equation}
Hereafter, the dependence on $U$ will be dropped from $B_{b}(U)$ and $d_U(x)$, whenever the graphon $U$ is clear from the context.

By Proposition~\ref{ppn:construction}, it suffices to only consider graphons $U$ satisfying $\E\left[I_p(p+U)\right] \lesssim p^\Delta I_{p}(1)$, where the hidden constant may depend on $H$ and $\delta$. The following consequences of this bound will be frequently used later on.
\begin{lem} \label{lem:apriori}
	Let $U$ be a graphon satisfying\footnote{More precisely, the statement is that for every constant $C > 0$ there is some constant $C'>0$ such that if \eqref{eq:apriori} holds with constant hidden $C$, then \eqref{eq:EU-upper}--\eqref{eq:low-deg-sq} all hold with hidden constant $C'$}
	\begin{equation} \label{eq:apriori}
	\E\left[I_p(p+U)\right] \lesssim p^\Delta I_{p}(1)\,.
	\end{equation}
	Then
	\begin{equation}\label{eq:EU-upper}
		\E[U] \lesssim p^{(\Delta+1)/2}\sqrt{\log(1/p)}\,,
	\end{equation}
	and 
	\begin{equation}\label{eq:EU2-upper}
		\E[U^2]  \lesssim p^\Delta\,,
	\end{equation}
	and furthermore $B_b = \{x : d(x) \ge b\}$, with $p = o(b)$, satisfies
	\begin{equation}\label{eq:B_b-upper}
	\lambda(B_b) \lesssim \frac{p^\Delta}{b}\,,
	\end{equation}
	where $\lambda$ denotes the Lebesgue measure, and, writing $\overline B_b := [0,1]\setminus B_b$,
	\begin{equation}\label{eq:low-deg-sq}
	\int_{\overline B_b}d(x)^2 \,\mathrm d x \lesssim p^\Delta b\,.
	\end{equation}
\end{lem}

We will prove Lemma~\ref{lem:apriori} shortly. The following estimates for $I_p(x)$ were given in~\cite{LZ-sparse}. The $\sim$ notation below is with respect to limits as $p\to 0$.
\begin{lem}[{\cite[Lemma~3.3]{LZ-sparse}}]\label{est1}
If $0\le  x \ll p$, then $I_p(p + x) \sim \frac12 x^2/p$, whereas when $p \ll x \le 1- p$ we have
$I_p(p + x) \sim x \log(x/p)$.
\end{lem}

\begin{lem}[{\cite[Lemma~3.4]{LZ-sparse}}]\label{est2} There is some constant $p_0 > 0$ such that for every $0 < p \le p_0$,
\[ I_p(p + x)\ge {(x/b)}^{2}I_p(p + b)\qquad\mbox{ for any $0\le x \le b\le 1-p- \log(1-p)$}\,.\]
\end{lem}

\begin{cor}[{\cite[Corollary 3.5]{LZ-sparse}}]\label{est3} There is some constant $p_0 > 0$ such that for every $0 < p \le p_0$,
\[I_p(p + x) \ge x^2I_p(1 - 1/ \log(1/p)) \sim x^2I_p(1)\qquad\mbox{ for any  $0 \le x \le 1 - p$}\,.\]
\end{cor}

As a consequence, observe that
\begin{equation}\label{eq-x32-bound}
x^{3/2} \lesssim I_p(p+x)/ I_p(1) + o(p^2)\quad\mbox{ for any $0\leq x \leq 1-p$}\,.
\end{equation}
Indeed, this is trivial for $x \ll p^{4/3}$ due to the $o(p^2)$ term;
if $p^{2/3} \leq x \leq 1-p$ then $I_p(p+x)  \gtrsim  x I_p(1)$ by Lemma~\ref{est1}; and in between, when $p^{4/3} \lesssim x \leq p^{2/3}$, we have 
\[
I_p(p+x) \stackrel{\text{Lem~\ref{est2}}}{\ge} (x/p^{2/3})^2 I_p(p+p^{2/3}) \gtrsim x^{3/2} p^{-2/3} I_p(p+p^{2/3}) \stackrel{\text{Lem~\ref{est1}}}{\gtrsim} x^{3/2} I_p(1)\,.
\]

\begin{proof}[Proof of Lemma~\ref{lem:apriori}]
From Lemma~\ref{est1} we have $I_p\big(p + a p^{(\Delta+1)/2}\sqrt{\log(1/p)}\big)\sim \frac12 a^2 p^\Delta I_p(1)$  for any $p\leq p_0$ and fixed $a\geq 0$ and $\Delta\geq 2$. However, $I_p(p+\E[U]) \leq \E\left[I_p(p+U)\right] \lesssim p^\Delta I_p(1)$ by the convexity of $I_{p}(\cdot)$ and~\eqref{eq:apriori}. Therefore, by the monotonicity of $I_p(p+x)$ for $x\geq 0$ we obtain the upper bound $\E[U] \lesssim p^{(\Delta+1)/2} \sqrt{\log(1/p)}$, proving \eqref{eq:EU-upper}. Finally, using \eqref{eq:apriori} and Corollary~\ref{est3} we obtain $\E[ U^{2} ] \lesssim \E [I_p(p+U)]/I_p(1)  \lesssim p^\Delta$, proving \eqref{eq:EU2-upper}.

By the convexity of $I_{p}(\cdot)$, for any $b \gg p$,
\[
\E\left[I_p(p+U)\right]
=\int_{[0,1]^2}I_p(p+U(x,y)) \, \mathrm d x\mathrm d y \ge \int_0^1 I_{p}(p+d_U(x)) \, \mathrm d x \ge \lambda(B_b)I_{p}(p+b)\,.
\]
It follows from  Lemma~\ref{est1} (combined with~\eqref{eq:apriori}) that for any $p\ll b \le 1-p$,
\[
\lambda(B_b) \le \frac{\E\left[I_p(p+U)\right]}{I_p(p+b)} \lesssim \frac{p^\Delta I_p(1)}{b\log(b/p)} \lesssim \frac{p^\Delta}{b}\,,
\]
proving \eqref{eq:B_b-upper}. Furthermore, by the convexity of $I_p(x)$ and Lemma~\ref{est2},
\begin{align*}
\E\left[ I_{p}(p+U)\right]\geq \int_{\overline B_b}I_{p}(p+d(x)) \,\mathrm d x\ge I_p(p+b) \int_{\overline B_b}\left(d(x)/ b\right)^2 \,\mathrm d x\,.
\end{align*}
Combining these, we get
\[
\int_{\overline B_b}d(x)^2 \,\mathrm d x \le \frac{b^2 \E\left[ I_p(p+U)\right]}{I_p(p+b)}\lesssim p^\Delta b\,,
\]
proving \eqref{eq:low-deg-sq}.
\end{proof}

\section{The triangle and the 4-cycle}
\label{sec:triangle-4cycle}

In the first part of this section, we recall, from \cite{LZ-sparse}, a short proof of Theorem~\ref{thm:main-graphon} for the triangle. In the second part, we prove it for the 4-cycle---which already illustrates the difficulties in extending the arguments of~\cite{LZ-sparse} to general graphs. A key new idea for the 4-cycle is to use an adaptively chosen degree threshold instead of a fixed threshold. This section is not needed for the proof of the general result, but it may be helpful in motivating the general analysis later on.

\subsection{The variational problem for the triangle}
\label{sec:triangles}
The case of $K_3$ (and larger cliques) was resolved in~\cite{LZ-sparse} via a divide-and-conquer approach: roughly put, by setting a certain \emph{degree threshold} $b=b(p)$ one finds that, in any graphon whose entropy is of the correct \emph{order}, the Lebesgue measure of the set $B_b$ of high degree points (defined in~\eqref{eq:highdeg}) asymptotically determines the surplus of $K_{1,2}$ copies (as in the anti-clique graphon), whereas the points in $\overline B_b$ are left only with the possibility of contributing extra triangles through ``cliques.'' We include a (slightly condensed) version of this proof, 
and explain why a more sophisticated cut-off $b(p)$ (tailored to each $U$) is needed for a general $H$.

\begin{thm}[{\cite[Theorem 2.2]{LZ-sparse}}] \label{th:triangle}
Fix $\delta >0$. As $p\to 0$,
\[
\phi(K_3, p, \delta)\sim\min\left\{\tfrac12 \delta^{2/3}\,,\tfrac13 \delta\right\}p^2\log (1/p)\,.
\]
\end{thm}
\begin{proof}
Let $W=p+U$ with $t(K_3, W)\geq (1+\delta)p^3$. Assume $U$ is nonnegative and  satisfies \eqref{eq:apriori} (or else we are done). Expanding $t(K_3, W)$ in terms of $U$,
\begin{equation}
t(K_3, W)-p^3=t(K_3, U)+3p\,t(K_{1, 2}, U)+3p^2\,\E[U]\geq \delta p^3\,.
\label{eq:triangle1}
\end{equation}
Now, $\E[U] = o(p)$ by~\eqref{eq:EU-upper},  and so~\eqref{eq:triangle1} reduces to
\begin{equation}
t(K_3, U)+3p\,t(K_{1, 2}, U)\geq \left(\delta -o(1)\right) p^3\,.
\label{eq:triangle2}
\end{equation}
Let $B_b=\{x: d(x)> b\}$ as in~\eqref{eq:highdeg}. By~\eqref{eq:B_b-upper}, for any $p \ll b \ll 1-p$ we have
\begin{equation}
\int_{B_b\times B_b}U(x, y)^2\,\mathrm d x\mathrm d y\leq \lambda(B_b)^2 \lesssim p^4/b^2 \ll p^2\,.
\label{Usmall}
\end{equation}
Let the degree threshold be some function $b=b(p)$ such that $\sqrt{p \log(1/p)} \ll b \ll 1$, and note that
\begin{equation}
\int_{[0,1]^3} U(x, y)U(y, z)U(y, z)\pmb 1\{x\in B_b \text{ or } y\in B_b \text{ or } z \in B_b\}\mathrm d x\mathrm d y\mathrm d z\leq 3\lambda(B_b)\E[U] \ll p^3\,,
\label{eq:1B}
\end{equation}
where the last step uses~\eqref{eq:B_b-upper} and~\eqref{eq:EU-upper}. Let
\begin{equation}
\theta_b := p^{-2}\int_{B_b\times \overline B_b}U(x,y)^2\,\mathrm d x\mathrm d y 
\quad\text{and}\quad 
\eta_b :=p^{-2}\int_{\overline B_b\times \overline B_b}U(x,y)^2\,\mathrm d x\mathrm d y\,.
\label{eq:theta-eta-2}
\end{equation}
We deduce from~\eqref{eq:1B} and generalized H\" older's inequality (Theorem~\ref{thm:holder}) that
\begin{align*}
t(K_3, U) &=
\int_{\overline B_b \times \overline B_b\times \overline B_b } U(x, y)U(y, z)U(x, z) \,\mathrm d x\mathrm d y\mathrm d z + o(p^3)
\\
&\le \biggl( \int_{\overline B_b \times \overline B_b} U \biggr)^{3/2} + o(p^3)
= \big(\eta_b^{3/2}+o(1)\big) p^3\,.	
\end{align*}
Similarly, by~\eqref{eq:low-deg-sq}, \eqref{Usmall}, and the Cauchy--Schwarz inequality, we obtain, for any $p\ll b\ll 1$,
\begin{equation}
t(K_{1, 2}, U)=\int_{B_b \times \overline B_b \times \overline B_b} U(x, y)U(x, z)
\,\mathrm d x\mathrm d y\mathrm d z+o(p^2)\leq \big(\theta_b + o(1)\big)p^2\,.
\label{eq:2star}
\end{equation}
Combining the above two inequalities with \eqref{eq:triangle2}, we obtain 
\[
3\theta_b+\eta_b^{3/2}\geq \delta-o(1)\,.
\]
By Corollary~\ref{est3},
\begin{align*}
\E\left[ I_{p}(p+U)\right]
&\ge (1-o(1))\left(\theta_b+\tfrac12\eta_b\right)p^2 \log(1/p)\,
\\
&\ge (1-o(1))  \hspace{-.7em}\min_{\substack{x,y \ge 0 \\ 3x + y^{3/2}\ge \delta}} \hspace{-.7em} ( x + \tfrac12 y) \hspace{.5em}  p^2 \log (1/p)
\\
&\sim \min\bigl\{ \tfrac12 \delta^{3/2}, \tfrac13\delta \bigr\} p^2 \log (1/p)\,,
\end{align*}
since the minimum is attained at either $x=0$ or $y=0$. This together with the clique and anti-clique constructions in \S\ref{sec:clique-anticlique} (recall that $P_{K_3}(x)=3x+1$) completes the proof for the case of triangles.
\end{proof}

\subsection{The variational problem for the 4-cycle}
\label{sec:4cycle} 
The argument in~\S\ref{sec:triangles} can be applied to other graphs $H$ and rule out certain subgraphs $F$ of $H$ from having a non-negligible contribution to $t(H,W)$ in the expansion analogous to~\eqref{eq:triangle1}. However,  as we next see, already for $H = C_4$ new ideas are required to tackle all subgraphs of $C_4$ and deduce the correct lower bound on $\phi(C_4,p,\delta)$.

Let $W=p+U$ with  $t(C_4, W)\geq (1+\delta)p^4$. As earlier, assume $U$ is nonnegative and satisfies \eqref{eq:apriori}. Expanding $t(C_4, W)$ as in \eqref{eq:triangle1},
\begin{equation}
\delta p^4 \le 
t(C_4, W)-p^4 = t(C_4, U)+4p^2\,t(K_{1, 2}, U) + 4p\,t(P_4, U) + 2p^2(\E U)^2 + 4 p^3\, \E U\,,
\label{eq:4cycle1}
\end{equation}
where $P_4$ is the path on 4 vertices and we used that $\E[U] \ll p$ from~\eqref{eq:EU-upper}. By~\eqref{eq:EU-upper}, $\E[U] = o(p)$, we the final two terms on the right are negligible.

Let $B_b=\{x: d(x)> b\}$ as in~\eqref{eq:highdeg}.
By generalized H\"older's inequality, embeddings $P_4 \mapsto (w,x,y,z)\in [0,1]^4$ with $x\in\overline B_b$ satisfy
\begin{align*}
&\hspace{-2em}\int_{[0, 1]\times \overline B_b \times [0,1] \times [0,1] } U(w, x)U(x, y)U(y, z) \,\mathrm d w\mathrm d x\mathrm d y\mathrm d z 
\\
&= \int_{\overline B_b \times [0, 1] \times [0,1]} d(x)U(x, y)U(y, z)\,\mathrm d x\mathrm d y \mathrm d z 
\\
&\leq  \left(\int_{\overline B_b} d(x)^2\,\mathrm d x\right)^{1/2}\left(\int_0^1 U(x, y)^2\,\mathrm d x\mathrm d y\right) \lesssim p^3\sqrt b \ll p^3
\end{align*}
by~\eqref{eq:low-deg-sq} and \eqref{eq:EU2-upper}, and using $b = o(1)$ for the last inequality. Hence, embeddings of $P_4$ with a non-negligible contribution to $t(P_4, U)$ must place both of the interior (degree 2) vertices in $B_b$.  The contribution from such embeddings is therefore at most $\lambda(B_b)^2 \lesssim p^4/b^2 \ll p^3$, provided $b\gg \sqrt p$. Hence, $t(P_4, U)  =o(p^3)$.

We have already encountered the term $t(K_{1,2},U)$ previously when analyzing $H = K_3$. So let us focus our attention on the term $t(C_4, U)$. For convenience, write
\[
\widetilde U(w, x, y, z) := U(w, x)U(x, y)U(y, z)U(z, w).
\]
By~\eqref{eq:B_b-upper} and~\eqref{eq:EU-upper},
\begin{equation}
\int_{B_b\times B_b \times [0,1] \times [0,1]} \widetilde{U}(w, x,y,z) \,\mathrm d w\mathrm d x\mathrm d y\mathrm d z \leq \lambda(B_b)^2\E[U] \ll p^4\,.
\label{eq:1B4cycle}
\end{equation}
So any embedding placing two consecutive vertices of $C_4$ in $B_b$ is negligible.
As in \eqref{eq:theta-eta-2}, set
\[
\theta_b := p^{-2}\int_{B_b\times \overline B_b}U(x,y)^2\,\mathrm d x\mathrm d y 
\quad\text{and}\quad 
\eta_b :=p^{-2}\int_{\overline B_b\times \overline B_b}U(x,y)^2\,\mathrm d x\mathrm d y\,.
\]
The three other possible embeddings of $C_4$ (see Fig.~\ref{fig:4cycle} for an illustration) are handled via generalized H\"older's inequality as follows.
\begin{figure}
\centering
\vspace{-0.3cm}
\begin{tikzpicture}
    \node (plot1) at (0,0) {
      \includegraphics[width=.8\textwidth]{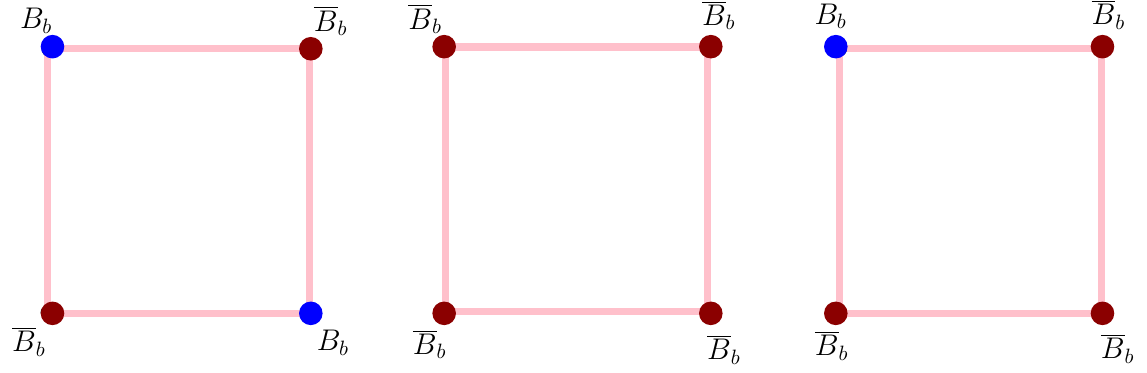}};
    \begin{scope}[shift={(plot1.south west)}]
    \node at (2.25,0) {(a)};
    \node at (7.1,0) {(b)};
    \node at (11.8,0) {(c)};
    \end{scope}
\end{tikzpicture}
\vspace{-0.6cm}
\caption{Different embeddings of the 4-cycle:  (a) and (b) are non-negligible.}
\label{fig:4cycle}
\vspace{-0.3cm}
\end{figure}
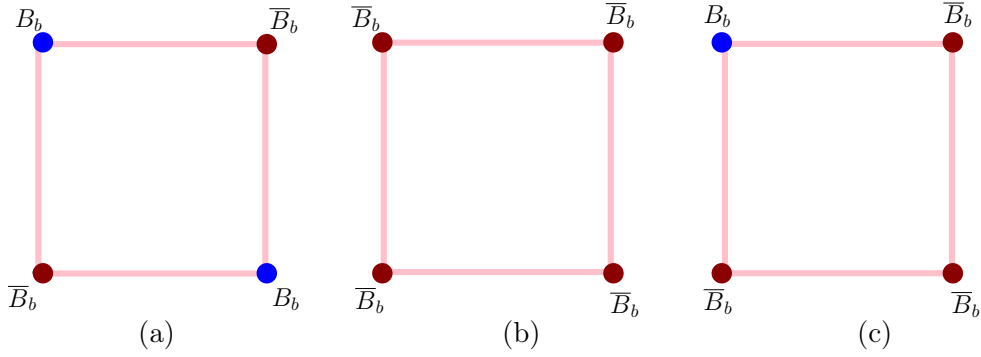
\begin{enumerate}[\indent(a)]
\item Two nonadjacent vertices in $B_b$ (2 configurations):
\begin{equation}
\int_{B_b\times\overline B_b\times B_b\times\overline B_b} \widetilde U(w, x, y, z) \, \mathrm d w\mathrm d x\mathrm d y\mathrm d z\leq
 \biggl(\int_{B_b\times\overline B_b}U^2\biggr)^2 \leq
\left(\theta_b^{2}+o(1)\right)p^4\,.
\label{eq:cycle4}
\end{equation}

\item No vertices in $B_b$ (1 configuration):
\begin{equation}
\int_{\overline B_b \times \overline B_b \times \overline B_b \times \overline B_b}  \widetilde U(w, x, y, z)
\,\mathrm d w\mathrm d x\mathrm d y\mathrm d z\leq
\biggl(\int_{\overline B_b\times\overline B_b} U^2\biggr)^2\leq
\left(\eta_b^{2}+o(1)\right) p^4\,.
\label{eq:cycle5}
\end{equation}

\item A single vertex in $B_b$ (4 configurations):
\begin{equation}
\int_{B_b\times\overline B_b \times \overline B_b \times \overline B_b }   \widetilde U(w, x, y, z)
\,\mathrm d w\mathrm d x\mathrm d y\mathrm d z
\le
\biggl(\int_{B_b\times\overline B_b} U^2\biggr)\biggl(\int_{\overline B_b\times\overline B_b} U^2\biggr)
\le
\left( \theta_b\eta_b+o(1)\right)p^4\,.
\label{eq:cycle6}
\end{equation}
\emph{(As we will see shortly, this final estimate is not tight.)}
\end{enumerate}
Combining~\eqref{eq:1B4cycle}, \eqref{eq:cycle4}--\eqref{eq:cycle6} and the estimate~\eqref{eq:2star} for $t(K_{1,2},U)$, the expansion~\eqref{eq:4cycle1} gives
\begin{equation}
2\theta_b^2+\eta_b^2+4\theta_b\eta_b+4\theta_b\geq \delta-o(1)\,,
\label{eq:4cycleholder}
\end{equation}
valid for any $\sqrt{p}\ll b\ll 1$.  As in the case of $K_3$, we wish to minimize $\theta_b+\frac12\eta_b$ subject to this constraint.
Unfortunately, the minimum of $\theta_b+\frac12\eta_b$ subject to~\eqref{eq:4cycleholder} is not attained at  $\theta_b=0$ or $\eta_b=0$. Thus the lower bound obtained in this way does not match the upper bound from Proposition~\ref{ppn:construction}.

Recall the  clique and anti-clique graphons in \S\ref{sec:graphon}. The main contribution of the anti-clique to $t(C_4,U)$ is through embeddings of type~(a), whereas for the clique it is through embeddings of type~(b). For the correct lower bound, we must show that the contribution from embeddings of type~(c) is negligible; however, this can no longer be achieved using any arbitrary $\sqrt{p}\ll b \ll 1$. To conclude the proof, we select the degree threshold $b$ \emph{adaptively} based on the graphon $U$.

\begin{lem}[Adaptive degree threshold for $C_4$] 
Assume that $U$ satisfies \eqref{eq:apriori}. There exists $b$ (possibly depending on $U$) with $\sqrt{p} \ll b \ll 1$ such that
\begin{equation}
\int_{B_b\times \overline B_b\times \overline B_b\times \overline B_b } \widetilde{U}(w,x,y,z)\,\mathrm d w \mathrm d x \mathrm d y \mathrm d z = o(p^4)\,.
\label{eq:adapC4}
\end{equation}
\label{lem:adaptive_C4}
\end{lem}
\begin{proof}
It suffices to show that for every constant $\varepsilon > 0$, we can find $b = b(U, p, \varepsilon)$ with $\sqrt{p} \ll b \ll 1$ such that 
\[
\int_{B_b\times \overline B_b\times \overline B_b\times \overline B_b } \widetilde{U}(w,x,y,z)\,\mathrm d w \mathrm d x \mathrm d y \mathrm d z \le \varepsilon p^4\,
\]
provided that $p$ is small enough. By generalized H\"older's inequality, $t(C_4, U) \le \E[U^2]^2 \lesssim p^4$ by \eqref{eq:EU2-upper}. So we can fix some constant $C >0$ such that $t(C_4, U) \le C p^4$. Set $ M = \lceil 2C/\varepsilon\rceil$.
Further let $\sqrt{p}\ll b_1 \ll b_2 \ll \cdots \ll b_M\ll 1$. As usual, set $B_{b_i}=\{x: d(x)> b_i\}$, and note that 
$B_{b_M}\subseteq \ldots \subseteq B_{b_1}$.
For every $2\leq i \leq M$, using~\eqref{eq:B_b-upper} and~\eqref{eq:low-deg-sq} we find that
\begin{align}
\int_{B_{b_i}\times \overline B_{b_i} \times \overline B_{b_{i-1}} \times \overline B_{b_i}}   \widetilde U(w, x, y, z)
\,\mathrm d w\mathrm d x\mathrm d y\mathrm d z & \leq \lambda(B_{b_i}) \int_{\overline B_{b_{i-1}}}\hspace{-.5em} d(x)^2 \,\mathrm d x
\lesssim \frac{b_{i-1}}{b_i} p^4 \ll p^4\,.
\label{eq:adaptive1}
\end{align}
Thus, setting $b = b_i$ works provided that
\begin{equation}
\int_{B_{b_i}\times \overline B_{b_i} \times (B_{b_{i-1}}\setminus B_{b_i})\times \overline B_{b_i}}   \widetilde U(w, x, y, z) \,\mathrm d w\mathrm d x\mathrm d y\mathrm d z \leq \tfrac12 \varepsilon p^4\,.
\label{eq:adaptive2}
\end{equation}
We finish the proof by observing that there is necessarily some $2\leq i \leq M$ satisfying~\eqref{eq:adaptive2}, as otherwise---since the sets $\{B_{b_{i}} \times \overline B_{b_i}\times  (B_{b_{i-1}} \setminus  B_{b_{i}}) \times \overline B_{b_i}: 2\leq i \leq M\}$ are mutually disjoint---we would get $t(C_4,U) > (M \varepsilon /2)p^4 \geq C p^4$, a contradiction to our choice of $C$.
\end{proof}
\begin{remark}
Note the advantage of using multiple thresholds with $b_{i-1}\ll b_i$ in the proof of Lemma~\ref{lem:adaptive_C4}: a single threshold function $b_{i}\equiv b$ (as in~\S\ref{sec:triangles}) would have given a bound of $O(p^4)$ for the left-hand of~\eqref{eq:adaptive1} vs.\ the sought $o(p^4)$. This idea will be crucial in our arguments for general graphs.
\end{remark}

Combining the above estimates, the expansion~\eqref{eq:4cycle1} implies (this is \eqref{eq:4cycleholder} with the extraneous $4\theta_b\eta_b$ term deleted)
\[ 
2\theta_b^2 + 4\theta_b + \eta_b^2 \geq \delta -o(1)\,.
\] 
By Corollary~\ref{est3},
\begin{align*}
\E\left[ I_{p}(p+U)\right]
&\ge (1-o(1))\left(\theta_b+\tfrac12\eta_b\right)p^2 \log(1/p)\,
\\
&\ge (1-o(1))  \hspace{-.7em}\min_{\substack{x,y \ge 0 \\ 2x^2 + 4x + y^2 \ge \delta}} \hspace{-.7em} ( x + \tfrac12 y) \hspace{.5em}  p^2 \log (1/p)\,.
\end{align*}
This minimum is attained at either $x=0$ or $y=0$ by the following lemma, thereby giving the bound for $\phi(C_4,p,\delta)$ matching the one from Proposition~\ref{ppn:construction} (recall that $P_{C_4}(x)=2x^2 + 4x + 1$).

\begin{lem}\label{lem:convarg} Let $f,g$ be convex  nondecreasing functions on $[0,\infty)$ and let $a>0$. The minimum of $x+y$ over the region
$\{x,y\geq0:f(x)+g(y)\ge a\}$ is attained at either $x=0$ or $y=0$.
\end{lem}

\begin{proof}
By convexity, if $\gamma=\frac{y}{x+y}$ then $f(x) \leq \gamma f(0)+ (1-\gamma)f(x+y)$ and $g(y) \leq (1-\gamma)g(0)+\gamma g(x+y)$, so
\begin{align*}
f(x)+g(y) 
&\le \gamma[f(0)+g(x+y)] + (1-\gamma)[f(x+y)+g(0)] 
\\
&\le \max\{f(0)+g(x+y),\, f(x+y)+g(0)\}\,.
\end{align*}
This shows that for a fixed value of $x+y$, $f(x) + g(y)$ is maximized at $x= 0$ or $y=0$. The claim then follows.
\end{proof}

\section{General graphs} \label{sec:general}

We begin the analysis for a general graph $H$. As always, $\Delta \ge 2$ denotes the maximum degree of $H$.

\subsection{Decomposition} \label{sec:decomp}
We can expand $t(H,W) = t(H,p+U)$ as
\begin{equation}
	\label{eq:H-expand}
t(H, W)-p^{|E(H)|} = \sum_{F} N(F,H) t(F, U)p^{|E(H)|-|E(F)|}\,,
\end{equation}
where the sum is taken over non-empty subgraphs $F$ of $H$ (up to isomorphism) and $N(F,H)$ is the number of subgraphs of $H$ isomorphic to $F$. Assuming $U$ satisfies \eqref{eq:apriori} (in particular the consequence \eqref{eq:EU2-upper} $\E[U^2] \lesssim p^\Delta$), every term on the right-hand side of \eqref{eq:H-expand} is of order $O(p^{|E(H)|})$, since by generalized H\"older's inequality, 
\begin{equation}
	\label{eq:F-density-order}
	t(F, U) 
	\le (\E[U^{\Delta}])^{|E(F)|/\Delta} 
	\le (\E[U^2])^{|E(F)|/\Delta} 
	\lesssim p^{|E(F)|}\,.
\end{equation}
However this bound is often not tight, as many contributions are negligible in that $t(F, U) = o(p^{|E(F)|})$, as we saw earlier in the case $H= K_3$ and $K_4$. We proceed by identifying and bounding the non-negligible terms.

\subsection{Negligible terms} \label{sec:negligible-terms}

Let $\tau(F)$ denote the minimum size of a vertex cover of $F$, where a \emph{vertex cover} of $F$ is a subset of vertices that intersects every edge of $F$.

\begin{lem}
	\label{lem:F-negligible}
	Let $\Delta \ge 2$ and $U$ be a graphon satisfying \eqref{eq:apriori}, i.e., $\E[I_p(p+U)] \lesssim p^\Delta I_p(1)$. Let $F$ be a connected graph with maximum degree at most $\Delta$. If $\tau(F) > |E(F)|/\Delta$ and $F$ is not $\Delta$-regular, there is some constant $\kappa = \kappa(F) > 0$ such that $t(F, U) \lesssim p^{|E(F)|+\kappa} = o(p^{|E(F)|})$.
\end{lem}

We will prove this lemma shortly. Note that every subgraph $F$ of $H$ satisfies $\tau(F) \ge |E(F)|/\Delta$. Due to the above lemma, the set of subgraphs
\begin{equation}\label{eq:sF-def}
\sF_H :=\left\{F : F \text{ is a non-empty subgraph of } H \text{ with } \tau(F) = |E(F)|/\Delta \right\}
\end{equation}
plays an important role.
Let us highlight some basic properties of $\sF_H$ for a connected graph $H$, all of which are easy to prove.
\begin{itemize}
	\item Every $F \in \sF_H$ is bipartite and has maximum degree exactly $\Delta$.
	\item If $S$ is a minimum size vertex cover of $F \in \sF_H$, then $S$ is an independent set and every vertex of $S$ has degree $\Delta$ in $F$; furthermore, $(S, V(F)\setminus S)$ forms a vertex bipartition of $F$. Conversely, any non-empty independent set $S$ of $H^*$ (i.e., an independent set of $H$ consisting of degree $\Delta$ vertices in $H$) gives rise to an $F \in \sF_H$ with $S$ as a minimum vertex cover by forming $F$ using the edges of $H$ incident to $S$.
	\item Every $F \in \sF_H$ has a unique minimum vertex cover except when $H$ is a regular bipartite graph, in which case $H \in \sF_H$ has two different minimum vertex covers
	(corresponding to two sides of the vertex partition; here we use that $H$ is connected).
	\item For a regular graph $H$, we have $H \in \sF_H$ if and only if $H$ is bipartite.
\end{itemize}

\begin{cor} \label{cor:reduced-expand}
	Let $H$ be a connected graph with maximum degree $\Delta \ge 2$, and $F$ a non-empty subgraph of $H$. Let $U$ be a graphon satisfying \eqref{eq:apriori}. Then $t(F, U) = o(p^{|E(F)|})$ unless $F \in \sF_H$ or $F$ is $\Delta$-regular (in the latter case necessarily $F = H$). Consequently, 
	\begin{equation} \label{eq:H-expand-reduced}
		t(H, W)-p^{|E(H)|}=\sum_{F \in \sF_H \cup \{H\}}N(F, H)t(F, U)p^{|E(H)|-|E(F)|}+o(p^{|E(H)|})\,.
	\end{equation}
	If $H$ is irregular, one can replace the ``$\sF_H \cup \{H\}$'' in the summation by simply ``$\sF_H$''.
\end{cor}

\begin{proof}
	Suppose $\tau(F) > |E(F)|/\Delta$ and $F$ is not $\Delta$-regular. Let $F_1, \dots, F_k$ be the connected components of $F$. Since $\tau(F) = \tau(F_1) + \dots + \tau(F_k)$, we see that some $F_i$, say $F_1$, satisfies $\tau(F_1) > |E(F_1)|/\Delta$. Furthermore this $F_1$ is not $\Delta$-regular, since $H$ has no $\Delta$-regular subgraphs other than possibly itself due to its connectedness, and we have ruled out the possibility of $F = H$ being $\Delta$-regular in the hypothesis. Therefore $t(F_1, U) = o(p^{|F_1|})$ by Lemma~\ref{lem:F-negligible} and $t(F_i, U) = O(p^{|F_i|})$ for $i \ge 2$ by \eqref{eq:F-density-order}. Therefore $t(F, U) = t(F_1, U) \dotsm t(F_k, U) = o(p^{|E(F)|})$ as claimed. The claim \eqref{eq:H-expand-reduced} is then deduced from \eqref{eq:H-expand}. 
\end{proof}

From the correspondence between $\sF_H$ and independent sets of $H^*$, we obtain
\begin{equation}
\label{eq:indpoly-N(F,H)}
P_{H^*}(x) = 1 + \sum_{F\in\sF_H} N(F,H) x^{|E(F)|/\Delta} + {\bf 1}\{H \text{ is regular and bipartite}\} x^{|E(H)|/\Delta}\,.
\end{equation}
In~\S\ref{sec:irregular} and \S\ref{sec:regular} we will relate each term $t(F,U)p^{|E(H)|-|E(F)|}$ in the right-hand of~\eqref{eq:H-expand-reduced} to $\theta^{|E(F)|/\Delta}  p^{|E(H)|}$ where $\theta$ is defined analogously to~\eqref{eq:theta-eta-2}.

\medskip

We will prove Lemma~\ref{lem:F-negligible} by using generalized H\"older's inequality with non-uniform weights. The \emph{fractional matching number}, denoted $\nu^*(F)$, is the maximum value of $w(E(F)) := \sum_{e\in E(F)} w(e)$ over all weight functions $w:E(F)\to[0,1]$ such that $\sum_{e \ni v} w(e)\leq 1$ for every vertex $v\in F$, i.e., $\nu^*(F)$ is the linear relaxation of the matching number $\nu(F)$, which corresponds to restricting $w(e)$ to $\{0,1\}$-values. Recall that the \emph{matching number} $\nu(F)$ is the size of the largest matching of $F$, where a \emph{matching} is a subset of edges with no two edges sharing a vertex.

One has $\nu(H) \le \tau(H)$ for every graph $H$. K\"onig's theorem tells us that this is always an equality when $H$ is bipartite.

\begin{namedthm}{K\"onig's theorem}
For every bipartite graph $H$, $\nu(H) = \tau(H)$.	
\end{namedthm}

\begin{lem} \label{lem:irreg-frac-matching}
Let $\Delta \ge 2$ and $F$ a connected graph with maximum degree at most $\Delta$ and not $\Delta$-regular. If $\tau(F) > |E(F)|/\Delta$, then
\[
\nu^*(F) > |E(F)|/\Delta\,.
\]
\end{lem}
\begin{proof}
We have $\nu(F) \le \nu^*(F) \le \tau(F)$ by linear programming duality. If $F$ is bipartite, then $\nu(F) = \tau(F)$ by K\"onig's theorem, so $\nu^*(F) = \tau(F) > |E(F)|/\Delta$ by hypothesis.

Now assume that $F$ is not bipartite, and let $C$ be an odd cycle in $F$. Let $u_0 \in V(F)$ with $\deg_F(u_0) < \Delta$. Let $P$ be a shortest path in $F$ from $u_0$ to $C$ (using that $F$ is connected) and let the vertices of $P$ be $u_0,u_1,\ldots,u_r$. Write the vertices of $C$ as $v_0,\ldots,v_{2k}$ with $v_0 = u_r$. We now perturb the constant weights $w\equiv 1/\Delta$ into new weights $w'$ by, for each $0 \le j \le r-1$, adding $(-1)^j \varepsilon$ to the weight of edge $(u_j,u_{j+1})$ along the path, and, for each $0 \le j \le 2k$, adding $(-1)^{j+r} \varepsilon/2$ to the weight of edge $(v_j,v_{j+1})$ along the cycle (index taken modulo $2k+1$).
Since $C$ is an odd cycle, $\sum_{e \ni v} w'(e) = \sum_{e \ni v}w(e)$ for all $v\neq u_0$. Thus, $w'$ is admissible as long as $\varepsilon$ is small enough.
Finally, if $r$ is even then $w'(P)=w(P)$ and $w'(C)=w(C)+\varepsilon/2$, and if $r$ is odd then $w'(P)=w(P)+\varepsilon$ and $w'(C)=w(C)-\varepsilon/2$. Either way, $\nu^*(F)\geq w'(E(F)) = E(F)/\Delta+\varepsilon/2$.
\end{proof}

\begin{lem} \label{lem:nonunif-holder}
Let $\Delta \ge 2$ and $U$ a graphon satisfying \eqref{eq:apriori}. Let $F$ be a graph with maximum degree at most $\Delta$. If $\nu^*(F) > |E(F)|/\Delta$, then there is some constant $\kappa = \kappa(F) > 0$ such that $t(F, U) \lesssim p^{|E(F)|+\kappa}$.
\end{lem}

\begin{proof}
Since $\nu^*(F) > |E(F)|/\Delta$, the function $w \equiv 1/\Delta$ is not a local maximum of the linear program defining $\nu^*(F)$.
If $\Delta=2$, one can perturb $w \equiv 1/2$ into $w'\colon E(F)\rightarrow [0, 2/3]$, such that $\sum_{v\ni e}w'(e) \leq 1$ and  $\sum_{e\in E(F)}w'(e) \ge (|E(F)|+\kappa) /2$ for some constant $\kappa>0$.  By generalized H\"older's inequality with these weights,
\begin{multline*}
t(F, U)
\le \prod_{e\in E(F)} \E[U^{1/w'(e)}]^{w'(e)} 
\le \E [U^{3/2}]^{(|E(F)|+\kappa)/2}
\\ 
\stackrel{\eqref{eq-x32-bound}}{\leq} \bigl(\E[I_p(p+U)]/I_p(1) + o(p^2)\bigr)^{(|E(F)|+\kappa)/2}
\stackrel{\eqref{eq:apriori}}{\lesssim} p^{|E(F)|+\kappa}\,.
\end{multline*}
If $\Delta\geq 3$, one can perturb $w\equiv 1/\Delta$ into $w'\colon E(F)\rightarrow [0, 1/2]$ such that $\sum_{v\ni e}w'(e) \leq 1$ and $\sum_{e\in E(F)}w'(e)\ge (|E(F)|+\kappa)/\Delta$ for some constant $\kappa >0$. By generalized H\"older's inequality with these weights,
\[
t(F, U)
\le \prod_{e\in E(F)} \E[ U^{1/w'(e)}]^{w'(e)}
\le \E[U^2]^{(|E(F)|+\kappa)/\Delta} 
\stackrel{\eqref{eq:EU2-upper}}{\lesssim} p^{|E(F)|+\kappa}\,. \qedhere
\]
\end{proof}

Lemma~\ref{lem:F-negligible} follows immediately by combining Lemmas~\ref{lem:irreg-frac-matching} and \ref{lem:nonunif-holder}.

\subsection{Main terms} \label{sec:main-terms}

We now state upper bounds to the non-negligible terms $t(F,U)$ in \eqref{eq:H-expand-reduced}, namely for $F \in \sF_H$, as well as $F = H$ in the case when $H$ is regular.

Recall from \eqref{eq:highdeg} that we denote the set of high degree vertices in $U$ by
\[
B_b =\{x:d(x)\ge b\}\,,\quad\text{ where }\quad d(x) =\int_0^1 U(x, y) \,\mathrm d y\,.
\]
Define (compared to \eqref{eq:theta-eta-2} we have $p^{-\Delta}$ here instead of $p^{-2}$)
\begin{equation}
\theta_b := p^{-\Delta}\int_{B_b\times \overline B_b}U(x,y)^2\mathrm d x\mathrm d y 
\quad\text{and}\quad 
\eta_b :=p^{-\Delta}\int_{\overline B_b\times \overline B_b}U(x,y)^2\mathrm d x\mathrm d y\,.
\label{eq:theta_eta-general}
\end{equation}

For any graph $F$, we write
\[
W(\vec x|F) := \prod_{(i, j)\in E(F)}W(x_i, x_j)\,,
\]
where $\vec x = (x_v)_{v \in V(F)}$ is clear from context. If the domain of an integral is omitted, then it is assumed to be $[0,1]$ for every $x_v$.

\begin{ppn} \label{ppn:main-term}
Let $\Delta \ge 2$ and $U$ be a graphon satisfying \eqref{eq:apriori}. 
\begin{enumerate}
	\item[(a)] Let $F$ be a connected irregular bipartite graph with maximum degree $\Delta$ and $\tau(F) = |E(F)|/\Delta$. Let $A$ be the unique vertex cover of $F$ with size $|E(F)|/\Delta$. Then, for any $p^{1/3} \ll b \ll 1$,
	\begin{align*}
	t(F,U) 
	&= \int U(\vec x | F) \mathbf 1 
	\left\{
	\begin{array}{ll}
	\forall v \in A \colon x_v \in B_b \\
	\forall u \notin A \colon x_u \in \overline B_b
	\end{array}
	\,
	\right\} \mathrm d \vec x
	+ o(p^{|E(F)|})
	\\
	& \le \theta_b^{|E(F)|/\Delta} p^{|E(F)|} + o(p^{|E(F)|}) \,.
	\end{align*}
	
	\item[(b)] Let $H$ be a connected $\Delta$-regular non-bipartite graph. Then there exists some constant $\kappa = \kappa(H) > 0$ such that for any $p^\kappa \le b \ll 1$ one has
	\begin{align*}
	t(H,U) 
	&= \int U(\vec x | H) \mathbf 1 
	\left\{
	\forall v \in V(H) \colon x_v \in \overline B_b
	\right\} \mathrm d \vec x
	+ o(p^{|E(H)|})
	\\
	& \le \eta_b^{|E(H)|/\Delta} p^{|E(H)|} + o(p^{|E(H)|}) \,.
	\end{align*}
	
	\item[(c)] Let $H$ be a connected $\Delta$-regular bipartite graph with vertex bipartition $(A, V(H)\setminus A)$. For any $b_0 = o(1)$, there exists some $b$ with $b_0 \le b \ll 1$ such that
	\[
	t(H, U) = \Gamma_1 + \Gamma_2 + \Gamma_3 + o(p^{|E(H)|})
	\]
	where
	\begin{align*}
	\Gamma_1 
	& = 
	\int U(\vec x | H)	
	\left\{
	\begin{array}{ll}
	\forall v \in A \colon x_v \in B_b \\
	\forall u \notin A \colon x_u \in \overline B_b
	\end{array}
	\,
	\right\} \mathrm d \vec x
	+ o(p^{|E(H)|})
	\\
	&\le \theta_b^{|E(H)|/\Delta} p^{|E(H)|} + o(p^{|E(H)|}) \,,
	\\
	\Gamma_2 
	&= 
	\int U(\vec x | H)	
	\left\{
	\begin{array}{ll}
	\forall v \in A \colon x_v \in \overline B_b \\
	\forall u \notin A \colon x_u \in  B_b
	\end{array}
	\,
	\right\} \mathrm d \vec x
	+ o(p^{|E(H)|})
	\\
	&\le \theta_b^{|E(H)|/\Delta} p^{|E(H)|} + o(p^{|E(H)|}) \,,
	\\
	\Gamma_3
	& = 
	\int U(\vec x | H)	
	\left\{ \forall v \in V(H) \colon x_v \in \overline B_b	\right\} \mathrm d \vec x
	+ o(p^{|E(H)|})
	\\
	&\le \eta_b^{|E(H)|/\Delta} p^{|E(H)|} + o(p^{|E(H)|}) \,.
	\end{align*}
\end{enumerate}
\end{ppn}

Note that each ``$\le$'' in the statement of the proposition above follows from the generalized H\"older's inequality. For example, in (a), one bounds the integral from above by
\[
\left(\int_{B_b \times \overline B_b} U^{\Delta}\right)^{|E(F)|/\Delta}
\le 
\left(\int_{B_b \times \overline B_b} U^{2}\right)^{|E(F)|/\Delta}
= \theta_b^{|E(F)|/\Delta} p^{|E(F)|}\,.
\]

We will prove part (a) of Proposition~\ref{ppn:main-term} in \S\ref{sec:irregular} and parts (b) and (c) in \S\ref{sec:regular}. Now we use the proposition to deduce the main result about the asymptotic solutions to the variational problem.

\begin{proof}[Proof of Theorem~\ref{thm:main-graphon} assuming Proposition~\ref{ppn:main-term}] The upper bound to $\phi(H,p,\delta)$ has been handled by the clique and anti-clique constructions. It remains to prove the lower bounds.

First we consider the case with $H$ being irregular, which only requires part (a) of Proposition~\ref{ppn:main-term}.
Set $b = p^{1/4}$. By applying Proposition~\ref{ppn:main-term}(a) to connected components of $F \in \sF_H$ (and noting \eqref{eq:F-density-order}), we have, by~\eqref{eq:H-expand-reduced},
\begin{align}
t(H, W)-p^{|E(H)|}
&= \sum_{F \in \sF_H} N(F, H)t(F, U)p^{|E(H)|-|E(F)|}+o(p^{|E(H)|})\,. \nonumber
\\
&\le p^{|E(H)|} \sum_{F\in \sF_H}N(F, H)\theta_b^{|E(F)|/\Delta} p^{|E(H)|} = p^{|E(H)|} \left(P_{H^*}(\theta_b)-1\right)\,.
\label{eq:rr2}
\end{align}
So $t(H, W) \ge (1+\delta)p^{|E(H)|}$ implies $P_{H^*}(\theta_b)\ge 1+\delta-o(1)$. Recall that $\theta$ satisfies $P_{H^*}(\theta) = 1+\delta$. So $\theta_b \ge \theta - o(1)$. By Corollary~\ref{est3},
\[
\E[I_{p}(p+U)] 
\ge (\theta_b-o(1))p^\Delta \log(1/p)
\ge (\theta - o(1))p^\Delta \log(1/p)\,.
\]
This completes the proof in the case of irregular $H$.

\medskip

Now assume that $H$ is regular.
Proposition~\ref{ppn:main-term} implies that there is some $b= o(1)$ such that 
\begin{equation}
	\label{eq:H-reg-t-upper}
t(H, W) \le p^{|E(H)|} (P_{H^*}(\theta_b) + \eta_b^{|E(H)|/\Delta} + o(1)).
\end{equation}
Indeed, we bound the terms in the expansion \eqref{eq:H-expand-reduced} of $t(H, W) = t(H,p+U)$ by applying Proposition~\ref{ppn:main-term} and then match them to the terms of $P_{H^*}$ in \eqref{eq:indpoly-N(F,H)}. As earlier, Proposition~\ref{ppn:main-term} is applied to the connected components of $F$; recall that since $H$ is connected, no subgraph other than itself can be $\Delta$-regular. It is also worth noting that if $H$ is bipartite and regular then $H \in \sF_H$ contributes twice in both \eqref{eq:H-expand-reduced} and \eqref{eq:indpoly-N(F,H)}.

So $t(H, W) \ge (1+\delta)p^{|E(F)|}$ implies $P_{H^*}(\theta_b) + \eta_b^{|E(H)/\Delta} \ge 1 + \delta - o(1)$, which, by Lemma~\ref{lem:convarg}, implies that $\theta_b + \tfrac12 \eta_b \ge \min\{\theta, \tfrac12 \delta^{2/|V(H)|}\} - o(1)$ as $P_{H^*}(\theta) = 1+\delta$. Thus, by Corollary~\ref{est3},
\begin{align*}
\E[I_{p}(p+U)] &\ge (\theta_b + \tfrac12 \eta_b-o(1))p^\Delta \log(1/p)
\\
&\ge (\min\{\theta, \tfrac12 \delta^{2/|V(H)|}\} - o(1))p^\Delta \log(1/p)\,.\end{align*}
Thereby completing the proof in the case of regular $H$.
\end{proof}

It remains to prove Proposition~\ref{ppn:main-term}, which will be done in the next two sections. Roughly, the idea is to eliminate negligible contributions from $t(F,U)$ (or $t(H,U)$ in parts (b) and (c)) by taking a certain subgraph $M$ of $F$ with maximum degree 2 (i.e., a disjoint union of cycles and paths), so that $U(\vec x | F) \le U(\vec x | M)$. This reduces the problem to paths and cycles. We then extend the analysis of the triangle and the 4-cycle in \S\ref{sec:triangle-4cycle} to handle these cases.

\section{Bounding contributions from irregular components}
\label{sec:irregular}

In this section we prove Proposition~\ref{ppn:main-term}(a). We first need a preparatory lemma.

\subsection{Existence of a 2-matching}

We say that a subset $M$ of edges of $F$ is a \emph{2-matching} if $M$ is a union of two matchings of $F$ (equivalently $M$ is a disjoint union of paths and even cycles).\footnote{This differs slightly from the notion of 2-matchings in the literature (cf.~\cite{LP86}), the definition here being a special case.} The following lemma is the main result of this section, and it will be used for proving Proposition~\ref{ppn:main-term}(a). The reader may wish to skip its proof on the first reading.

\begin{lem} \label{lem:2-matching}
	Let $F$ be a connected irregular bipartite graph with maximum degree $\Delta$ and $\tau(F) = |E(F)|/\Delta$. For every vertex $v$ of $F$, there is a 2-matching $M$ of $F$ of size $2|E(F)|/\Delta$ such that the connected component of $v$ in $M$ is a path.
\end{lem}

Recall that a \emph{proper edge-coloring} is a coloring of edges so that edges that share a common vertex receive different colors.
The following result is classic. See \cite[Theorem~1.4.18]{LP86} for a proof via embedding the graph in a larger $\Delta$-regular graph.

\begin{namedthm}{K\"onig's edge-coloring theorem}
	Every bipartite graph of maximum degree $\Delta$ has a proper edge-coloring with $\Delta$ colors.
\end{namedthm}

\begin{cor} \label{cor:matching-max-deg-verts}
	Every bipartite graph has a maximum matching which covers all maximum degree vertices.
\end{cor}

\begin{proof}
	Let $G$ be the bipartite graph and $\Delta$ its maximum degree. By K\"onig's edge-coloring theorem, $G$ has a proper edge-coloring of the graph with $\Delta$ colors. Every degree $\Delta$ vertex is incident to all $\Delta$ colors. Let $M$ denote the edges of an arbitrary color class. Then $M$ is a matching that covers all degree $\Delta$ vertices. If $M$ is not a maximum matching, then we can repeatedly replace $M$ by a larger matching via ``augmenting paths'' \cite[Theorem~1.2.1]{LP86}, while maintaining the property that all degree $\Delta$ vertices are covered by $M$. The process terminates with a maximum matching $M$ that covers all degree $\Delta$ vertices.
\end{proof}

Returning to Lemma~\ref{lem:2-matching}, note that if we apply K\"onig's edge-coloring theorem to $F$ and take $M$ to be the union of two arbitrary color classes, then $M$ is a 2-matching of $F$ of size $2|E(F)|/\Delta$. It remains to modify $M$ so that the connected component of $v$ in $M$ is a path.

\begin{proof}[Proof of Lemma~\ref{lem:2-matching}]
The result is easy when $\Delta = 2$, in which case we can take $M = F$. So assume $\Delta \ge 3$ from now on. 

Let $(A,B)$ be a vertex bipartition of $F$. Due to $F$ being connected and $\tau(F) = |E(F)|/\Delta$, either $A$ or $B$ must be unique minimum vertex cover of $F$. Relabeling if necessary, assume that $A$ is the minimum vertex cover. So $\tau(F) = |E(F)|/\Delta = |A| < |B|$ as $F$ is irregular, and $\deg_F(v) = \Delta$ for every $v \in A$.

First consider the case $v \in B$. We will show that in fact we can have $\deg_M(v) = 1$. If $\deg_F(v) < \Delta$, then by K\"onig's edge-coloring theorem, the edges of $F$ can be partitioned into $\Delta$ matchings of size $|A|$. We obtain the desired $M$ by taking two such matchings, with one matching covering $v$ and the other matching not covering $v$ (we can do this since $1 \le \deg_F(v) < \Delta$).

Now suppose $v \in B$ and $\deg_F(v) = \Delta$. Since $F$ is connected and irregular, and $\deg_F(a) = \Delta$ for every $a\in A$, we see that every $S \subseteq A$ has at least $|S| + 1$ neighbors in $B$ (otherwise we would have a $\Delta$-regular component). Let $F - v$ denote $F$ with $v$ removed (along with edges incident to $v$). Apply Hall's matching theorem to $F-v$ and we find that the maximum matching of $F-v$ has size $|A|$. Then Corollary~\ref{cor:matching-max-deg-verts} applied to $F-v$ gives a matching $M_1$ of size $|A|$ in $F$ not covering $v$ but covering every other vertex of degree $\Delta$ in $B$. Let $F'$ denote $F$ with $M_1$ removed. We claim that $F'$ has a matching of size $|A|$, since otherwise K\"onig's theorem would imply that $F'$ has a vertex cover of size $|A| - 1$, which is impossible since $F'$ has $|A|(\Delta-1)$ edges, and all its vertices have degree at most $\Delta - 1$ with the exception of $v$, which has degree $\Delta$. Thus, by Corollary~\ref{cor:matching-max-deg-verts} again, we can find a matching $M_2$ in $F$ of size $|A|$ that covers $v$. Taking $M = M_1 \cup M_2$ works.

Finally, suppose $v \in A$. Let $u \in B$ be an arbitrary neighbor of $u$. By above, there is a 2-matching $M$ of size $2|A|$ such that $\deg_M(u) = 1$. If $(u,v) \in M$, then we are done. Otherwise, let $w$ be an arbitrary neighbor of $v$ in $M$, and modify $M$ by removing $(v,w)$ and adding $(u,v)$. Then $v$ lies on a path component in the modified $M$.
\end{proof}

\subsection{Proof of Proposition~\ref{ppn:main-term}(a)} The claim follows from the next two lemmas.

\begin{lem} \label{lem:adj-neg}
Let $\Delta \ge 2$ and $U$ be a graphon satisfying \eqref{eq:apriori}.
Let $F$ be a bipartite graph with maximum degree $\Delta$ and $\tau(F) = |E(F)|/\Delta$. Let $A$ be its unique vertex cover of size $|E(F)|/\Delta$. As long as $b \gg p^{1/3}$,
\[
\int U(\vec x|F) {\mathbf{1}}\{ \exists (i,j)\in E(F)\colon x_i, x_j \in B_b\}\mathrm d \vec x = o(p^{|E(F)|})\,.
\]
\end{lem}

\begin{proof}
Fix an edge $(i, j)\in E(F)$. It suffices to prove that 
\[
\int U(\vec x|F) {\mathbf{1}}\{x_i, x_j \in B_b\}\mathrm d \vec x = o(p^{|E(F)|})\,.
\] 
By K\"onig's edge-coloring theorem, there is a proper edge-coloring of $F$ with $\Delta$ colors, so that each color class is a matching of size exactly $|E(F)|/\Delta$ (since every vertex in $A$ must see all edge-colors). Thus there exists a 2-matching $M$ of $F$ of size $2|E(F)|/\Delta$ such that $(i, j)\in E(M)$. Let $M'$ be obtained from $M$ by removing the edges incident to either $i$ or $j$. So $2 \le |M \setminus M'| \le 3$ (since one of $i$ and $j$ has degree $\Delta$ in $H$, and hence degree $2$ in $M$). Hence
\begin{align*}
\int U(\vec x|M) {\mathbf{1}}\{x_i, x_j \in B_b\}\,\mathrm d \vec x
&\le \lambda(B_b)^2\int U(\vec x|M')\,\mathrm d \vec x
\le \lambda(B_b)^2 \ \E [U^2]^{|E(M')|/2}
\\
&\lesssim \left(\frac{p^\Delta}{b}\right)^2 p^{\frac{\Delta(|E(M)|-3)}{2}}\lesssim b^{-2}p^{|E(F)|+\frac{\Delta}{2}} \ll p^{|E(F)|}\,.
\end{align*}
The second inequality is by generalized H\"older's inequality. The next is due to \eqref{eq:B_b-upper} and \eqref{eq:EU2-upper}. Finally, since $U(\vec x|F) \leq  U(\vec x|M)$, the claim follows.
\end{proof}

\begin{lem}\label{lem:irreg-neg-A-low}
Let $\Delta \ge 2$ and $U$ be a graphon satisfying \eqref{eq:apriori}. 
Let $F$ be a connected irregular bipartite graph with maximum degree $\Delta$ and $\tau(F) = |E(F)|/\Delta$. Let $A$ be its unique vertex cover of size $|E(F)|/\Delta$. Then, as long as $b = o(1)$, one has
\[
\int U(\vec x|F){\mathbf{1}}\{\exists v \in A\colon x_v \in \overline B_b\}\mathrm d\vec x = o(p^{|E(F)|})\,.
\]
\end{lem}

Fix $v \in A$. It suffices to show that
\[
\int U(\vec x|F){\mathbf{1}}\{x_v \in \overline B_b\} \, \mathrm d\vec x = o(p^{|E(F)|})\,.
\]
Let $M$ be a 2-matching of $F$ of size $2|A|$ such that the connected component of $v$ in $M$ is a path. The existence of this 2-matching is guaranteed by Lemma~\ref{lem:2-matching}. 
Let $M_1, M_2, \ldots, M_q$ be the connected components of $M$, labeled so that $M_1$ is the connected component of $v$ in $M$. In particular, $M_1$ is a path. Note that  $U(\vec x | F) \le U(\vec x | M)$ and
\begin{equation} \label{eq:2-component-holder}
	t(M_i, U) \le \E[U^2]^{|E(M_i)|/2} \lesssim p^{\Delta|E(M_i)|/2}\, \text{ for each } 1 \le i \le q
\end{equation}
by generalized H\"older's inequality and \eqref{eq:EU2-upper}. Lemma~\ref{lem:irreg-neg-A-low} is reduced to proving
\[
\int U(\vec x|M_1){\mathbf{1}}\{x_v \in \overline B_b\} \, \mathrm d\vec x = o(p^{\Delta|E(M_1)|/2})\,,
\]
which follows from the next lemma.

\begin{lem} 
Let $\Delta \ge 2$ and $\ell$ be positive integers, and $U$ a graphon satisfying \eqref{eq:apriori}. Let $P$ denote a path on $2\ell+1$ vertices labeled $1, 2, \dots, 2\ell+1$. Then, for every $1 \le k \le \ell$, as long as $b = o(1)$,
\begin{equation}
	\label{eq:path-neg}
\int U(\vec x|P) \mathbf 1 \{x_{2k} \in \overline B_b\} = o(p^{\Delta \ell}).
\end{equation}
\end{lem}

\begin{proof}
Since the left-hand side of \eqref{eq:path-neg} cannot decrease as $b$ gets larger, we may assume that $b \gg p$.

We use induction on $k$. For $k = 1$, we have, by 
generalized H\"older's inequality followed by \eqref{eq:EU2-upper} and \eqref{eq:low-deg-sq},
\begin{align*}
\int U(\vec x|P) \mathbf 1 \{x_2 \in \overline B_b\} \, \mathrm d \vec x
&= \int \mathbf 1\{x_2 \in \overline B_b\} d(x_2) U(x_2, x_3)\dotsm U(x_{2\ell},x_{2\ell+1}) \, \mathrm d\vec x
\\
&\le \biggl(\int_{\overline B_{b}} d(x_1)^2 \, \mathrm d x_1\biggr)^{1/2} \E [U^2]^{(2\ell-1)/2}
\\
&\lesssim (p^\Delta b)^{1/2}  p^{\Delta (2\ell-1)/2}
\\
&= \sqrt{b} p^{\Delta \ell} = o(p^{\Delta \ell})\,.
\end{align*}

Now let us prove claim for $k \ge 2$ assuming that
\begin{equation}
	 \label{eq:path-neg-induction}
\int U(\vec x|P) \mathbf 1 \{x_{2k-2} \in \overline B_{b'}\} = o(p^{\Delta \ell})
\end{equation}
holds for any $b'=o(1)$, and in particular, for $b' = b^{1/3}$. By removing vertex $2k-2$ from $P$ and then applying generalized H\"older's inequality followed by Lemma~\ref{lem:apriori}, we have, for any $p \ll b',b \ll 1$,
\begin{align*}
&\hspace{-2em}\int U(\vec x | P) \mathbf 1\{x_{2(k-1)} \in B_{b'}, x_{2k} \in \overline B_b\} \, \mathrm d\vec x 
\\
&\le \lambda(B_{b'}) \int U(x_1,x_2) \dotsm U(x_{2k-4},x_{2k-3}) \cdot \mathrm 1\{x_{2k} \in \overline B_b\} U(x_{2k-1},x_{2k}) U(x_{2k},x_{2k+1}) \dotsm U(x_{2\ell},x_{2\ell+1}) \, \mathrm d \vec x 
\\
&\le \lambda(B_{b'}) \int U(x_1,x_2) \dotsm U(x_{2k-4},x_{2k-3}) \cdot \mathrm 1\{x_{2k} \in \overline B_b\} d(x_{2k}) U(x_{2k},x_{2k+1}) \dotsm U(x_{2\ell},x_{2\ell+1}) \, \mathrm d \vec x 
\\
&\le \lambda(B_{b'}) \biggl(\int_{\overline B_b} d(x_{2k})^2 \, \mathrm d x_{2k}\biggr)^{1/2} \E [U^2]^{(2\ell-3)/2}
\\
&\lesssim \frac{p^\Delta}{b'} (p^\Delta b)^{1/2} p^{\Delta (2\ell-3)/2}
= \frac{\sqrt{b}}{b'} p^{\Delta \ell}\,,
\end{align*}
which is $o(p^{\Delta \ell})$ if $b' = b^{1/3}$. Combining with \eqref{eq:path-neg-induction}, we obtain \eqref{eq:path-neg}. This completes the induction step.
\end{proof}

We have completed the proof of our main result, Theorem~\ref{thm:main-graphon}, for irregular graphs $H$.

\section{Bounding contributions from regular graphs}
\label{sec:regular}

\subsection{Proof of Proposition~\ref{ppn:main-term}(b)} Here $H$ is a connected $\Delta$-regular non-bipartite graph. Fix $v \in V(H)$.
Let $H-v$ denote $H$ after deleting $v$ (and all edges incident to $v$). We see that $H-v$ is not $\Delta$-regular and satisfies $\tau(H-v) > |E(H-v)|/\Delta$ (or else putting $v$ back in would give $\tau(H) = |E(H)|/\Delta$, which is impossible as $H$ is non-bipartite). Thus by Lemma~\ref{lem:F-negligible}, there is some constant $\kappa > 0$ such that $t(H-v, U) \lesssim p^{|E(H-v)| + \kappa} = p^{|E(H)| - \Delta + \kappa}$ (here we apply Lemma~\ref{lem:F-negligible} to some connected component $H'$ of $H-v$ satisfying $\tau(H')>|E(H')|/\Delta$ and use \eqref{eq:F-density-order} to bound the other components; note that $H$ has no $\Delta$-regular subgraphs other than itself due to its connectedness). Thus
\[
	\int U(\vec x|H) \mathbf 1 \{x_v \in B_b\}\, \mathrm d \vec x 
	\le \lambda(B_b) t(H-v, U)
	\lesssim \frac{p^\Delta}{b} p^{|E(H)| -\Delta + \kappa} = o( p^{|E(H)|})\,,
\]
provided that $b \gg p^{\kappa}$. Proposition~\ref{ppn:main-term}(b) then follows after considering all $v \in V(H)$.

\subsection{Proof of Proposition~\ref{ppn:main-term}(c)} Here $H$ is a connected $\Delta$-regular bipartite graph. Proposition~\ref{ppn:main-term}(c) is an immediate consequence of Lemma~\ref{lem:adj-neg} and the following lemma.

\begin{lem} \label{lem:reg-neg-2neighbor}
Let $\Delta \ge 2$ and $U$ be a graphon satisfying \eqref{eq:apriori}.
Let $H$ be a $\Delta$-regular bipartite graph. For any $b_0 = o(1)$, there exists some $b$ with $b_0 \le b \ll 1$ such that
\[
\int U(\vec x | H) \mathbf 1 \{\exists v \in V(H), u, w \in N_H(v) \colon x_u \in B_b, x_w \in \overline B_b\} \, \mathrm d \vec x = o(p^{|E(H)|})\,,
\]
where $N_H(v)$ is the neighborhood of the vertex $v$ in $H$.
\end{lem}

To deduce Proposition~\ref{ppn:main-term}(c) for a connected graph $H$, note that Lemma~\ref{lem:reg-neg-2neighbor} implies that to estimate $t(H, U)$ up to $o(p^{|E(H)|})$, one only needs to consider embeddings of $H$ where the vertices on the same side of the bipartition of $H$ get mapped to the same choice of $B_b$ versus $\overline B_b$. Furthermore, the case of both sides getting mapped to $B_b$ is eliminated by Lemma~\ref{lem:adj-neg}. The possibilities of $B_b$ versus $\overline B_b$ are captured by $\Gamma_1$, $\Gamma_2$, and $\Gamma_3$ in Proposition~\ref{ppn:main-term}(c).

\medskip

Now we prove Lemma~\ref{lem:reg-neg-2neighbor}. By K\"onig's edge-coloring theorem, $H$ has a proper edge-coloring with exactly $\Delta$ colors, where every color class is a perfect matching. For each path $u,v,w$, by taking the union of the two color classes that contain edges $(u,v)$ and $(v,w)$, we obtain a 2-matching containing the path $u,v,w$, which must be a disjoint union of cycles. Note that any cycle $C_\ell$ in $H$ satisfies \begin{equation} \label{eq:cycle-holder}
t(C_\ell, U) \le \E [U^2]^{\ell/2} \lesssim p^{\ell \Delta/2}
\end{equation}
by generalized H\"older's inequality and \eqref{eq:EU2-upper}.
As in Lemma~\eqref{lem:irreg-neg-A-low}, by isolating the cycle in the 2-matching that contains the vertices $u,v,w$, Lemma~\ref{lem:reg-neg-2neighbor} follows from the following claim. Its proof is an extension of the 4-cycle case in \S\ref{sec:4cycle}.

\begin{lem} \label{lem:cycle-neg}
	Let $\Delta \ge 2$ and $L$ be positive integers, and $U$ a graphon satisfying~\eqref{eq:apriori}. For any $b_0 = o(1)$, there exists some $b$ with $b_0 \le b \ll 1$ such that
	\[
	\int U(\vec x | C_\ell) \mathbf 1\{x_1 \in B_b, x_3 \in \overline B_b\} \, \mathrm d\vec x = o(p^{\Delta\ell/2})
	\]
	uniformly for all $3 \le \ell \le L$, where the vertices of the cycle $C_\ell$ are labeled $1, 2, \dots, \ell$ in cyclic order.
\end{lem}

\begin{proof}
Fix $\varepsilon \ge 0$ (which can be made arbitrarily small). It suffices to show that one can find $b$ (depending on $\varepsilon$, $L$ and $U$) with $b_0 \le b \ll 1$ such that
\begin{equation}
	\label{eq:cycle-neg-eps}
	\int U(\vec x | C_\ell) \mathbf 1\{x_1 \in B_b, x_3 \in \overline B_b\} \, \mathrm d\vec x \le (1+o(1))\varepsilon p^{\Delta\ell/2} 
\end{equation}
uniformly for all $3 \le \ell \le L$.

By removing the vertex labeled 1 and then applying generalized H\"older's inequality followed by Lemma~\ref{lem:apriori}, we have, for any $p \ll b',b'' \ll 1$,
\begin{align}
&\hspace{-2em}\int U(\vec x | C_\ell) \mathbf 1\{x_1 \in B_{b'}, x_3 \in \overline B_{b''}\} \, \mathrm d\vec x 
\nonumber
\\
&\le \lambda(B_{b'}) \int  \mathbf 1\{x_3 \in \overline B_{b''}\}  U(x_2,x_3)U(x_3,x_4)\dotsm U(x_{\ell-1},x_\ell) \, \mathrm d \vec x 
\nonumber
\\
&\le \lambda(B_{b'}) \int \mathbf 1\{x_3 \in \overline B_{b''}\} d(x_3)U(x_3,x_4)\dotsm U(x_{\ell-1},x_\ell) \, \mathrm d \vec x 
\nonumber
\\
&\le \lambda(B_{b'}) \biggl(\int_{\overline B_{b''}} d(x_3)^2 \, \mathrm d x_3\biggr)^{1/2} \E [U^2]^{(\ell-3)/2}
\nonumber
\\
&\lesssim \frac{p^\Delta}{b'} (p^\Delta b'')^{1/2} p^{\Delta (\ell-3)/2}
\lesssim \frac{\sqrt{b''}}{b'} p^{\Delta \ell/2}\,.
\label{eq:cycle2b}
\end{align}

Fix some $\ell$ for now. By \eqref{eq:cycle-holder}, there is some constant $C$ such that $t(C_\ell, U) \le C p^{\ell\Delta/2}$. Let $M:= \lceil C/\varepsilon \rceil$. Since it never hurts to make $b_0$ larger, assume that $p \ll b_0 \ll 1$. For any sequence $b_0 < b_1 < \cdots < b_M = o(1)$ with $b_{i-1} \le b_i^3$ for each $i \le M$, there is some $1 \le i \le M$ such that
\[
\int U(\vec x | C_\ell) \mathbf 1 \{x_1 \in B_{b_i}, x_3 \in B_{b_{i-1}} \setminus B_{b_i}\} \, \mathrm d \vec x
\le \varepsilon p^{\ell\Delta /2}\,
\]
since otherwise the sum of these integrals over $1 \le i \le M$ (note that the sets $B_{b_{i-1}}\setminus B_{b_i}$, $1 \le i \le M$, are disjoint) would violate $t(C_\ell, U) \le C p^{\ell\Delta/2}$. Combining the above estimate with \eqref{eq:cycle2b} applied with $b' = b_i$ and $b'' = b_{i-1}$ (so that $\sqrt{b''}/b' = o(1)$), we see that \eqref{eq:cycle-neg-eps} holds with $b = b_i$ (for this specific $\ell$).

To deduce \eqref{eq:cycle-neg-eps} for all $3 \le \ell \le L$, we start with the sequence $b_0 < b_1 < \dots < b_{M^L}$ satisfying $b_{i-1}\le b_i^3$ for each $i$ (e.g., take $b_i=b_0^{3^{-i}}$). Iteratively, for each $3 \le \ell \le L$, use the above argument to take a subsequence keeping at least $1/M$ fraction of the terms so that $b=b_i$ satisfies \eqref{eq:cycle-neg-eps} for this $\ell$ for every $b_i$ in the remaining subsequence. At the end of the process, we find some $b = b_i$ that satisfies \eqref{eq:cycle-neg-eps} for all $3 \le \ell \le L$.
\end{proof}

\begin{remark}
As mentioned below Theorem~\ref{thm:discrete-var}, when $H$ is regular and $n^{-2/\Delta}\ll p \ll n^{-1/\Delta}$ (a range in which the anti-clique construction is no longer applicable), it follows from our arguments that
\begin{align*}
\lim_{n \to \infty} \frac{\phi(H,n,p,\delta)}{n^{2}p^{\Delta}\log(1/p)} =\tfrac12\delta^{2/|V(H)|}\,.
\end{align*}
Indeed, the upper bound follows from the clique construction in Proposition~\ref{ppn:construction}(a), and it remains to verify the lower bound. Since $np^{\Delta} = o(1)$, in the above arguments, we can choose our threshold $b$ such that $np^{\Delta} \ll b \ll 1$, and thus $\lambda(B_b) \lesssim p^\Delta / b \ll n^{-1}$ by~\eqref{eq:B_b-upper}. Since $n\lambda(B_b)$ must be an integer in the discrete setting, we must have $\lambda(B_b) =0$, which rules out the anti-clique construction, thereby proving the claim.
\end{remark}

\section{Disconnected graphs}\label{sec:disconnected}

The arguments at the end of \eqref{sec:main-terms} can be easily modified to handle disconnected graphs by considering the different connected components. The solution of the variational problem can be expressed as a two variable constrained optimization problem.

\begin{thm} \label{thm:disconnected}
Let $H$ be a graph with maximum degree $\Delta \ge 2$. Let  $H_1, H_2, \ldots, H_s$ be the connected components of $H$. For fixed $\delta > 0$, we have
\[
\lim_{p\rightarrow 0} \frac{\phi(H, p, \delta)}{p^\Delta\log(1/p)}= 
\inf_{\theta, \eta\geq 0}\biggl\{\theta+\tfrac12 \eta : \prod_{i=1}^s\left(P_{H_i^*}(\theta)+{\bf 1}\{H_i\text{ is $\Delta$-regular}\} \eta^{|E(H_i)|/\Delta}\right)= 1+\delta \biggr\}\,.
\]
\end{thm}

For disconnected graphs the solution of the variational problem might not be attained by the clique or the anti-clique graphons, but by a mixture of these two.

\begin{example}Let $H$ be the disjoint union of a triangle ($K_3$) and a 2-star ($K_{1, 2}$). By Theorem~\ref{thm:disconnected},
\[
\lim_{p\rightarrow 0} \frac{\phi(H, p, \delta)}{p^2\log(1/p)}=\inf_{\theta, \eta\geq 0}\left\{\theta+\tfrac12 \eta : (1+3\theta+\eta^{3/2})(1+\theta)= 1+\delta \right\}\,.
\]
The pure clique construction corresponds to setting $\theta = 0$, so that $\eta = \delta^{2/3}$ and $\theta + \tfrac12 \eta = \tfrac12 \delta^{2/3}$. The pure anti-clique construction corresponds to setting $\eta = 0$, so that $\theta + \tfrac12 \eta = \theta \sim \tfrac{1}{\sqrt{3}} \delta^{1/2}$ for large $\delta$.
For large $\delta$, the optimal solution is asymptotically given by a mixture with $\theta \sim 3^{-3/5} \delta^{2/5}$ and $\eta \sim 3^{2/5} \delta^{2/5}$, giving $\theta + \tfrac12 \eta \sim \frac{5}{2 \cdot 3^{3/5}} \delta^{2/5}$.
\end{example}

\section*{Acknowledgment} This work was initiated when the first author was an intern at the Theory Group of Microsoft Research, Redmond. E.L.\ was supported in part by NSF grant DMS-1513403 and Y.Z.\ was supported by a Microsoft Research Ph.D.\ Fellowship. We thank the anonymous referee for helpful comments that greatly improved the exposition of the paper.

\bibliographystyle{amsplain_mod2}
\bibliography{ldp_ref}

\end{document}